\documentclass[11pt,leqno]{amsart}

\usepackage{amsmath}
\usepackage{amsthm}
\usepackage{graphicx}
\usepackage{mathtools}
\usepackage{amsfonts}
\usepackage{amssymb}
\usepackage{enumitem}
\usepackage{hyperref}
\usepackage{bm}
\usepackage{url}
\usepackage[margin=1 in]{geometry}
\usepackage{enumerate}
\usepackage[normalem]{ulem}
\usepackage{tikz}
\usepackage{tikz-cd}
\usepackage{xcolor}
\usetikzlibrary{matrix,arrows,positioning,automata}
  \usepackage{cancel}
  \usepackage{todonotes}
\usepackage{cite}
\usepackage[toc]{appendix}
  
\usepackage{comment}

\numberwithin{equation}{section}
\newcommand{\N}{\mathbb N}

\newcommand{\R}{\mathbb R}

\newcommand{\sP}{\mathcal{P}}

\def\XXint#1#2#3{{\setbox0=\hbox{$#1{#2#3}{\int}$}
\vcenter{\hbox{$#2#3$}}\kern-.5\wd0}}

\newcommand{\T}{\mathbb{T}}

\numberwithin{equation}{section}
\newtheorem{thm}{Theorem}[section]
\newtheorem{lem}[thm]{Lemma}
\newtheorem{cor}[thm]{Corollary}
\newtheorem{prop}[thm]{Proposition}
\newtheorem{assumption}[thm]{Assumptions}
\theoremstyle{definition}
\newtheorem{defn}[thm]{Definition}
\newtheorem{rmk}[thm]{Remark}

\def\smallnegint{\mathop{\int\mkern-13mu
        \raise.5ex\hbox{${\scriptscriptstyle\diagup}$}}\nolimits}

\def\div{\operatorname{div}}

\def\ssetminus{\,\raise.4ex\hbox{$\scriptstyle\setminus$}\,}

\newcommand{\be}{\begin{equation}}
\newcommand{\ee}{\end{equation}}

\newcommand{\bc}{\begin{case}}
\newcommand{\ec}{\end{cases}}
\newcommand{\bs}{\begin{split}}
\newcommand{\es}{\end{split}}

\newcommand{\norm}[1]{\left\Vert#1\right\Vert}

\renewcommand{\d}{d}

\newcommand{\Rd}{{\mathbb{R}^\d}}

\renewcommand{\bar}{\overline}
\renewcommand{\tilde}{\widetilde}

\newcommand{\del}{\partial}
\newcommand{\Id}{\mathrm{Id}}
\newcommand{\Lip}{\mathrm{Lip}}
\newcommand{\loc}{\mathrm{loc}}

\newcommand{\mcl}{\mathcal}

\newcommand{\oo}{\infty}
\newcommand{\Tan}{\mathrm{Tan}}

\title[On Viscosity Solutions of Hamilton-Jacobi Equations in the Wasserstein space]{On Viscosity Solutions of Hamilton-Jacobi Equations in the Wasserstein space and the Vanishing Viscosity Limit}

\date{\today}
\author[Giacomo Ceccherini Silberstein]{Giacomo Ceccherini Silberstein  }
\author[Daniela Tonon]{Daniela Tonon}


\address[1,2]{\emph{E-mail addresses}: ceccheri@math.unipd.it, tonon@math.unipd.it \newline \emph{Affiliation}: Dipartimento di Matematica ``Tullio Levi Civita'', Università di Padova, 
Via Trieste 63, 35121, Italy}

\begin{document}
\begin{abstract}
The aim of this article is twofold. First, we develop a unified framework for viscosity solutions to both first-order Hamilton-Jacobi equations and semilinear Hamilton–Jacobi equations driven by the idiosyncratic operator, defined on the Wasserstein Space. Second, we establish a vanishing-viscosity limit—extending beyond the classical control-theoretic setting—for solutions of semilinear Hamilton–Jacobi equations, proving their convergence to the corresponding first-order solution as the idiosyncratic noise vanishes. Our approach provides an optimal convergence rate.

We also present some results of independent interest. These include existence theorems for the first-order equation, obtained through an appropriate Hopf–Lax representation, and a useful description of the action of the idiosyncratic operator on geodesically convex functions.
\end{abstract}
\maketitle

\tableofcontents

\section{Introduction}

The starting objective of this work is to treat in a unified approach the following two Hamilton-Jacobi equations (HJE)s defined on the Wasserstein space $\mathcal{P}_2(\R^d)$: the first order HJE
\begin{equation}\label{HJBDet}
    -\del_t U(t,\mu) + \int_{\Rd} H(x, \partial_\mu U(t,\mu,x), \mu)d\mu(x) = 0 \quad \text{in } [0,T] \times \mcl P_2(\R^d),
\end{equation}
and the  semilinear HJE with idiosyncratic noise
\begin{equation}\label{HJBId}
-\partial_t U(t,\mu)
\;+\;
\int_{\mathbb{R}^d} 
H\bigl(x, \partial_\mu U(t,\mu,x), \mu\bigr)\, d\mu(x)
\;-\;
\int_{\mathbb{R}^d} 
\mathrm{tr}\!\left[ a(\mu,x)\, \nabla \partial_\mu U(t,\mu,x) \right] d\mu(x)
= 0,
\end{equation}
in $[0,T] \times \mathcal{P}_2(\mathbb{R}^d)$, where, $T>0$, $H: \R^d \times \R^d\times \mathcal{P}_2(\R^d)\to \R$ is the Hamiltonian, and $a: \mathcal{P}_2(\R^d)\times \R^d \to M_{d\times d}(\R)$ is the diffusion matrix that will be assumed to be positive definite. Usually we will consider the final time condition $U(T,\cdot)=\mathcal{G}(\cdot)$ where $\mathcal{G}:\mathcal{P}_2(\R^d)\to \R$ is the terminal datum. Here, $\partial_\mu U(t,\mu,x)$ stands for the Wasserstein gradient of $U$ at the point $(t,\mu,x)$.

In recent years, a viscosity-solution theory has been developed to address HJEs in the space of probability measures, ensuring well-posedness for this class of nonlinear equations. Classically, uniqueness and stability follow from the comparison principle, while the existence theory has been mainly established in the convex setting, identifying the unique viscosity solution with the value function of an optimal control problem on the space of probability measures. This framework yields the rigorous connection between HJEs on $\mathcal{P}_2(\R^d)$ and Mean Field Control problems.

In this paper, we develop an approach that enables us to treat these two equations using a unified definition of a viscosity solution. In particular, we establish an asymptotic relation between the viscosity solutions of \eqref{HJBDet} and \eqref{HJBId} through a vanishing-viscosity limit. To the best of our knowledge, this is the first work to provide such a perturbative link without imposing regularity assumptions on the functions $U^{\varepsilon}$, the viscosity solutions to \eqref{HJBId} when $a(\mu,x)=\varepsilon\text{Id}_{d\times d}$, apart from those inherent to the control setting. In \cite{tangpi2022probabilisticapproachvanishingviscosity}, a vanishing viscosity result is obtained also in the presence of a common noise operator, in the convex case; however, that analysis relies on additional regularity assumptions on the solutions.

\subsection{Overview on the existing literature}
The study of HJEs on Wasserstein spaces has gained a lot of interest, in these last years, thanks to their connections with Mean Field Control, Mean Field Games, and the associated Master Equation. 
Mean Field Games theory has been introduced by Lasry and P.L. Lions \cite{LLMeanFG}, and independently by Huang, Caines and Malhamé \cite{MeanFieldCaines}. It has been subsequently developed in Lions' courses at  Collège de France \cite{College} and it has attracted the interest of a large group of researchers. The well-posedness theory for solutions of HJEs in the space of probability measures is one of the objectives of this theory due to its connection with mean-field particle approximation \cite{cardaliaguet2022regularityvaluefunctionquantitative, cecchin2025quantitativeconvergencemeanfield}, the control of Mckean-Vlasov equation \cite{CarmonaDelarue, CarmonaDelarueLachapelle2013}, and large deviation theory \cite{FengKurtz2006}.

In the study of HJEs in $\mathcal{P}_2(\R^d)$,  various notions of viscosity solutions have been proposed in the literature (possibly in presence of a common noise) with the shared aim of extending the classical finite-dimensional framework to this more complex setting. However, this new context presents significant challenges. The geometry of $\mathcal{P}_2(\R^d)$ endowed with the Wasserstein distance $W_2$, originally exploited in the seminal work by Otto in \cite{Otto}, and later developed in the monograph by Ambrosio et al. \cite{AGS}, exhibits a singular and non linear structure. As a result, performing differential calculus in $\mathcal{P}_2(\R^d)$ involves the non-trivial task of developing a suitable notion of differentiation in a non-smooth realm. In \cite{AGS}, a subdifferential calculus has been introduced, allowing to give a finer description of the Gradient Flows in the space of probability measures. In particular, superdifferentiability of the squared Wasserstein distance is established. We remark that the lack of (local) differentiability of the distance is a common issue in the theory of HJEs in infinite dimensional spaces. The infinite dimensional theory has been firstly studied by Crandall and P.L. Lions in the series of papers \cite{CrandallLions1984,CrandallLions1986a,CrandallLions1986b,CrandallLions1990,CrandallLions1991}, where, to treat HJEs in Banach spaces, it is assumed the structural hypothesis of the existence of a smooth metric induced by the assigned norm.

To overcome these difficulties in the Wasserstein setting, several strategies have been explored.
\begin{itemize}

\item \emph{Penalization via entropy terms}. 
The singular geometry of the space is reflected by the non-differentiability of the squared Wasserstein distance $W_2$ at ``non-regular" measures (see \cite[Section ~2]{gigc}).
As described by Daudin and Seeger in \cite{DS}, by Bertucci and P.L. Lions in \cite{BertucciApprox}, and in the series of papers of the second author et al. \cite{CKTExi},\cite{CKTComp}, wellposedness for semilinear HJEs in the space of probability measures can be obtained by adding an entropic penalization that localizes the equation at points of differentiability of the squared  distance. The  semilinearity assumption  is fundamental since it helps to absorb the variational contribution of the entropy penalization in the HJE. These approaches are reminiscent of those introduced by Tataru \cite{Tataru1992,Tataru1994} to treat singular HJBs in Banach spaces, and the one developed by Feng and Katsoulakis \cite{FengKatsoulakis2009} to handle gradient flows in infinite dimensions.

Extending the theory to general, even degenerate and not variational, parabolic equation is an interesting problem out of the scope of this work. 

\item \emph{Regularization via weaker metrics}. 
The lack of regularity of the distance can be addressed by replacing it with a more regular alternative, following a program analogous to that developed in Banach spaces. One possibility is to modify the geometric structure of the space by introducing weaker yet sufficiently smooth metrics, as discussed by Soner and Yan in \cite{STorus}. A similar idea appears in Cosso et al. \cite{cosso2022masterbellmanequationwasserstein}, where a gauge function is employed and common noise is also handled. Along the same lines, Bayraktar et al.  \cite{bayraktar2023comparisonviscositysolutionsclass} introduce a clever change of variables that allows one to treat common noise as a finite-dimensional operator, to which an Ishii–Lions lemma (see \cite{crandall1992usersguideviscositysolutions}) can be applied. However, this approach has a drawback: the conditions required to guarantee the comparison principle become difficult to satisfy. Indeed, the weaker the metric, the more regularity is demanded of the data. This trade-off is absent in the penalization approach, where the assumptions on the Hamiltonian closely match those in the finite-dimensional setting.

\item \emph{Lifting to a Hilbert space}. 
In the courses at the Collège de France \cite{College}, P.L. Lions proposed to lift the HJE to a more familiar and smoother space, such as a Hilbert space, where an extensive viscosity solution theory has been developed in the 80's and 90's. Under this viewpoint, the Wasserstein space is realized as a (metric) submersion, and can be regarded as the space of laws of squared-integrable random variables.  This lift creates a bridge between the intrinsic calculus developed in \cite{AGS} and extrinsic calculus via the so called Lions' derivative. See also Carmona and Delarue \cite[Chapter ~5]{CarmonaDelarue}, Bertucci \cite{BSOT}, and Gangbo and Tudorascu \cite{GANGBO2019119}, where the use of Lions' lift to treat HJE over $\mathcal{P}_2(\R^d)$ has been deeply treated. For HJEs without idiosyncratic noise terms, this lift provides a natural strategy to give a well-posed notion of viscosity solution in the space of probability measures \cite{BSOT}. However, as already described in \cite[Rmk~2.20]{DS}, this approach does not adapt well to treat idiosyncratic noise, due to the singularity of the lifted idiosyncratic operator.

\item \emph{Viewing the Wasserstein space as a Positively Curved (PC)  space.} Some authors have instead chosen to develop a more general and abstract theory of HJEs in merely (lenght) metric spaces, see \cite{AMBFeng,GangboSwiech2015_metricViscosity}. Although this approach is elegant, it restricts the class of admissible Hamiltonians to the ones depending only on the metric slope, therefore not sensible to directional variations of the unknown function.  A further step was to observe that $(\mathcal{P}_2(\R^d),W_2)$ is a PC space, see \cite[Appendix]{AGS}. This geometric property allows one to define a notion of tangent bundle endowed with a weak scalar product. A test-function approach can then be employed to obtain well-posedness. See also\cite{JerhaouiAussedatZidani2024} for an application of this construction. As observed in \cite{DS}, the link between test functions and the subdifferential approach is not yet understood. For this reason, in the idiosyncratic case we adopt a more classical notion of solution, leaving the investigation of the relationship between the two frameworks to future work.
\item \emph{Superdifferentiability of the Wasserstein distance}. In \cite{BSOT}, the author realizes that a superdifferentiability of the Wasserstein distance is sufficient to establish a comparison principle for \eqref{HJBDet}. In particular, he proposes a notion that ensures both the superdifferentiability of the squared Wasserstein distance and existence results based on the Hilbertian lift.
\end{itemize}

\subsection{Our main results and contributions.}

We describe here our main contributions, referring to Section \ref{s:Set Up} for the notations used.
\begin{itemize}
\item[(1)] \textbf{Unifying notion of superdifferential.}

In this work, we provide a careful analysis of the weak Riemannian structure in the Wasserstein space, in order to select a superdifferential $\partial^{+}$ with the following properties:
\begin{itemize}
    \item[(i)] The Wasserstein distance squared is superdifferentiable, following the approach in \cite{BSOT}.
    \item [(ii)] $\partial^{+}$ is weaker than the definition of superdifferential given in \cite{BSOT}.
    \item[(iii)] $\partial^{+}$ coincides with the superdifferential adopted in \cite[Definition ~2.5]{DS} when restricted to measures that are absolutely continuous w.r.t. the Lesbesgue measure.
\end{itemize}
The appropriate notion of superdifferential $\partial^{+}$ is
$$
\partial^{+}=\partial_{S}^{+}\cap \text{Tan}^{K}_{\mu}\mathcal{P}_2(\R^d),
$$
where $\partial_S$ is the the strong differential, see Definition \ref{d: SDiff1}, and $\text{Tan}^{K}_{\mu}\mathcal{P}_2(\R^d)$ is the \emph{geometric} Tangent space at a  measure $\mu\in \mathcal{P}_2(\R^d)$, see Definition \ref{d: TanK}
Indeed, $\partial^{+}$ is a subset of the superdifferential defined in \cite{BSOT}. On the other hand, in $\mathcal{P}_{2,\text{ac}}(\R^d)$, the set absolutely continuous measures w.r.t. the Lesbesgue measure, the baricentric projection $\text{Pr}$, see \eqref{Pr}, yields $\forall \mu \in \mathcal{P}_{2,\text{ac}}(\R^d)$
$$
    (\Tan^{K}_{\mu}\mathcal P_2(\R^d),W_\mu)\simeq_{\text{Pr}}(\Tan^{M}_{\mu}\mathcal P_2(\R^d),\|\cdot\|_{L^2_\mu(\R^d)})= (\overline{\left\{ \nabla \phi \colon \phi \in C^{\infty}_c(\R^d) \right\}}^{L^2_\mu(\R^d)}, \|\cdot\|_{L^2_{\mu}(\R^d)}).
$$
This identification allows to relate our notion of superdifferential with the one introduced in \cite{DS}, namely:
\begin{equation*}
    \partial^{+}_S\phi(\mu)\cap \Tan^K_{\mu}\mathcal{P}_2(\R^d)\simeq_\mathrm{Pr}\partial^{+}_{M}\phi(\mu)\cap \Tan^{M}_{\mu}\mathcal{P}_2(\R^d), \quad \forall \mu \in \mathcal{P}_{2,\text{ac}}(\R^d).
\end{equation*}
We  also refer to Lemma \ref{Strong=Opt} for more details. Although this result is not new in the literature, see \cite[Theorem ~3.14]{GANGBO2019119}, we provide  here an alternative proof using the weak Riemannian formalism introduced in \cite{GigliPhd} and adopted in this work.

Moreover, in Proposition \ref{prop:super}  we show that Wasserstein distance squared is superdifferentiable w.r.t to $\partial^+$. This property ensures the uniqueness results for the equations \eqref{HJBDet} and \eqref{HJBId}, relying on the previous contributions \cite{BSOT} and \cite{DS}.

This unified approach allow us to compare, in the last part of this paper, the solutions of \eqref{HJBDet} and \eqref{HJBId} by means of a doubling variables procedure.

\item[(2)]\textbf{Hopf Lax Formula.}

With the aim of  providing a self-contained existence result, that justifies our notion of viscosity solution, we consider, see Subsection \ref{existence},  the function
 $$U(t,\mu)=\inf_{\gamma \in \mathcal{P}_{2,\mu}(\R^{2d})}\left\{ \mathcal{G}(\exp_{\mu}(\gamma)) + (T-t)\mathcal{L}\left(\frac{1}{T-t}\cdot\gamma\right) \right\},
 $$
where $\mathcal{L}: \mathcal{P}_2(\R^{2d})\to \R$ is the relaxed Lagrangian defined by
$$\mathcal{L}(\gamma)= \int_{\R^{2d}}L(v)d\gamma(x,v),$$
and $L:\R^d \to \R$ is the Lagrangian associated to the Hamiltonian $H$. Moreover, $\exp_{\mu}: \mathcal{P}_{2,\mu}(\R^{2d})\to \mathcal P_2(\R^d)$ denotes the exponential map in the Wasserstein space, defined on  $\mathcal{P}_{2,\mu}(\R^{2d}):=\left\{\gamma \in \mathcal{P}_2{(\R^{2d})}, (\pi_1)_{\#}\gamma=\mu \right\}$.

By showing  that the above formula  provides a representation of the viscosity solution of \eqref{HJBDet} with convex Hamiltonian $H(x,p,\mu)=H(p)$ and a not necessarily  differentiable Lagrangian, we generalize the existence result established in \cite{BSOT}.

A related Hopf-Lax representation formula was previously introduced in \cite[Section ~5]{GANGBO2019119}, where first-order time dependent HJEs are considered in the case the Hamiltonian $H$ is convex and satisfies a structural decreasing property, see equation (5.7) there. This condition ensures compatibility between their lifted Hamiltonian and the Wasserstein geometry. In that setting, the Hopf-Lax representation, see for instance \cite[Section 5.4]{GANGBO2019119}, is established in the quadratic case, relying crucially on the aforementioned decreasing property satisfied by the Hamiltonian. By contrast, our result does not require any such structural decreasing assumption of $H$, see also Remark \ref{decreasing}.

\item[(3)]\textbf{Geodesic Convexity and Idiosyncratic Operator.}
In Section \ref{s: section3}, we inspection the action of the idiosyncratic operator on geodesically convex functions. To  this end, we prove a connection between geodesic convexity and the convexity of the flat derivative w.r.t. the $\R^d$ variable.
Indeed, we prove the following convexity result on mixtures of measures, which, to the best of our knowledge, is new in the literature.\\

\noindent \textbf{Lemma \ref{geodesic convexity}.}\emph{ Suppose $d\geq2$.
Let $\mathcal{F}: \mathcal{P}_2(\R^d)\to \R$ be a continuous $\Lambda$-geodesically convex function. Let $\mu\in \mathcal P_2(\R^d)$, then $\forall h \in [0,1]$ the function 
    $$\mathcal P_2(\R^d) \ni \nu \mapsto \mathcal{F}((1-h)\mu+ h\nu)$$
is $h\Lambda$-geodesically convex.\\
}

The main difficulty in proving this result is that, even if $t \mapsto \nu_t$ is a $W_2$-constant speed geodesic, the curve $(1-h)\mu+h\nu_t$ is not, in general. Nevertheless, this curve is still induced by a transport coupling, so that using \cite[Theorem 9.1]{CavSavSod23}, where the equivalence between geodesic convexity and total convexity is shown, we obtain a convex like inequality along this curve, which yields the desired conclusion.
We remark that  the dimensional restriction $d\geq 2$ comes from \cite[Theorem 9.1]{CavSavSod23}. Relying on this result, we do not know if our result can be extended to $d=1$. 

As a byproduct of the previous lemma, we obtain the following result.\\

\noindent \textbf{Corollary \ref{geodesic convexity coro}.}
\emph{
Suppose $d\geq 2$.
Let $\mathcal{F}: \mathcal{P}_2(\R^d)\to \R$ be a continuous $\Lambda$-geodesically convex function. If $\mathcal{F}$ is flat differentiable at $\mu\in \mathcal P_2(\R^d)$ then $\R^d\ni x \mapsto \mathcal{D}_{\mu}\mathcal{F}(\mu,x)$  is $\Lambda$-convex.  \\
}

In the above, we denoted by $\mathcal{D}_{\mu}$ the flat derivative operator at $\mu$. The proof of the  corollary relies on the well known representation formula
\begin{equation*}
    \mathcal{D}_{\mu}\mathcal{F}(\mu,x):=\lim_{h \to 0}\frac{\mathcal{F}((1-h)\mu+ h \delta_x)-\mathcal{F}(\mu)}{h} \quad (\mu,x)\in \mathcal{P}_2(\R^d)\times \R^d.
\end{equation*}

The above results are crucial to prove an estimate for the  weak action of the idiosyncratic operator introduced in \cite{DS}.\\

\noindent \textbf{Proposition \ref{generalconcavityestimate}.}
\emph{Suppose $d \geq 2$. Let $\mu \in \mathcal{P}_2(\R^d)$ s.t. $\mathcal{I}(\mu)<\infty$, $\mu^{-1}\in L^{\infty}_{\text{loc}}(\R^d)$ and $\mu \in W^{1,\infty}_{\text{loc}}(\R^d)$. Let $\mathcal{F}: \R^d \to \R$ be a $\Lambda$-geodesically concave functional that is flat differentiable at $\mu\in\mathcal P_2(\R^d)$. Then 
\begin{equation*}\label{weakaction}
    -\int_{\R^d}\big<\nabla \log \mu,\nabla \mathcal{D}_{\mu}\mathcal{F}(\mu,x)\big>d\mu\leq d\Lambda.
\end{equation*}
}

\item[(4)]\textbf{Vanishing viscosity limit.}

The final result constitutes the main motivation of this article. Namely, showing the convergence of  $U^{\varepsilon}$, the viscosity solution of \eqref{HJBId},  with $a(\mu,x)=\varepsilon\text{Id}_{d\times d}$, towards $U$ viscosity solution of \eqref{HJBDet}, providing an optimal rate of convergence.\\

\noindent\textbf{Theorem \ref{vanishing}.}\emph{
Suppose Assumptions \ref{a: DataHyp} hold. Let $U^{\varepsilon}$ and $U$ be two bounded viscosity solutions of  \eqref{HJBId} and \eqref{HJBDet}, respectively, s.t. that $ U^{\varepsilon}(T,\mu)=U(T,\mu)=\mathcal{G}(\mu)$.
Additionally, suppose $U^{\varepsilon}$ to be uniformly in time $W_1$-Lipschitz continuous for each $\varepsilon$, and $U$ to be uniformly in time $W_2$-Lipschitz continuous. Moreover, suppose $U$ to be jointly continuous with respect to the product topology of
$[0,T]\times\mathcal P_2(\mathbb R^d)$,
where $\mathcal P_2(\mathbb R^d)$ is endowed with the weak convergence as in Definition \ref{weakconvergence}.
Then, there exists a constant $C\geq 0$ depending only on the data s.t. 
\begin{equation*}
\sup_{(t,\mu)\in[0,T]\times \mathcal{P}_2(\R^d)} \vert U^{\varepsilon}(t,\mu)- U(t,\mu) \vert \leq C\sqrt{\varepsilon}.
\end{equation*}
}
We begin by briefly discussing the assumptions. The boundedness requirement on the solutions can be relaxed by introducing a suitable growth control; however, we decided not to focus on this technical issue. The main restriction of our result is the further continuity assumption w.r.t. the weak convergence of measures that we briefly discuss here: the space $(\mathcal{P}_2(\R^d),W_2)$ is not locally compact, see \cite[Remark ~7.1.9]{AGS}, and this usually constitutes an issue in any optimization procedure, as in the doubling variables argument. However, $\mathcal{P}_2(\R^d)$ is locally compact w.r.t. the weak convergence, recalled in Subsection \ref{sectionweakconvergence}. We therefore impose further regularity w.r.t. this convergence in order to properly optimize. See also the discussion in \cite[Section ~4.6]{bertucci2025doublingvariablestechniqueorder} where tentative approaches to solve this important problem are also proposed.

We now heuristically outline the main arguments in the lighter time-independent case, since the core argument is completely analogous to that of the evolutive case.
Note that, to derive the vanishing-viscosity limit, we do not rely on possible explicit representations of the solution, as is often done in the control-theoretic setting. Instead, we adopt a purely PDE-based approach, following methods already established in the finite-dimensional case, see the monograph \cite[Section~VI]{Bardi1997OptimalCA} or \cite[Theorem 5.1]{CrandallLionsApproximation}. The main motivation for this choice is that this approach may also be useful in more general settings, such as the broader metric framework considered in \cite{CKTComp}.

Then the strategy closely follows the finite dimensional case. We perform a doubling variables argument penalizing the difference between the two solutions:
\begin{equation*}
    (\mu,\nu) \mapsto U^{\varepsilon}(\mu)-U(\nu) -\frac{1}{2\alpha}W_2^2(\mu,\nu) -\delta \mathcal{E}(\mu) + \mbox{localizing term},
\end{equation*}
where $\mathcal{E}: \mathcal{P}_2(\R^d) \to \R \cup\{+ \infty\}$ is the entropy functional relative to the $d$-dimensional Lesbesgue measure $\lambda_d$.
Thanks to this penalization, we obtain a maximum $(\bar \mu,\bar \nu)\in \mathcal{P}_2(\R^d)\times \mathcal{P}_2(\R^d)$ at which we can exploit the joint \emph{superdifferentiability }of $W^2_2(\cdot , \cdot)$ at $(\bar \mu, \bar \nu)$ (here we use the compactness ensured by the continuity w.r.t. the weak convergence of the two solutions, indeed  the $W_1$-Lipschitz continuity of $U^\varepsilon$ implies its continuity w.r.t. weak convergence).  As usual, this procedure gives elements in the subdifferential of $U^{\varepsilon}-\delta \mathcal{E}$ at $\bar \mu$ and in the superdifferential of $U$ at $\bar \nu$, which can be used to apply the viscosity solution property to $U^\varepsilon$ and $U$. Let us also stress that the maximization gives $\mathcal{I}(\bar \mu)< \infty$. We then apply the  viscosity solution property which yields
\begin{equation*}
    U^{\varepsilon}(\bar \mu)-U(\bar \nu)\leq C\alpha+ K(\varepsilon)\delta -\underbrace{\varepsilon\int_{\R^d}\big<\nabla \log \bar \mu(x), \frac{x-T_{\bar \mu}^{\bar \nu}(x)}{\alpha}\big>d\bar \mu}_{I_{\varepsilon}}(x),
\end{equation*}
where the constant $C$ depends only on the local Lipschitz regularity of $H$. The term $C\alpha$ is obtained similarly  in the proof of the comparison principle presented in \cite{DS}. The main issue in our case,  is the dependence of the constant $K(\varepsilon)$ on the diffusion, which quantitatively compensates the fact that we considered an element of the superdifferential of $U^{\varepsilon} -\delta \mathcal{E}$, and not merely of $U^{\varepsilon}$. Indeed, this constant is expected in general to blowup, see the formal computation in \cite[Section ~2.3, Eq. (2.23)]{DS}. Therefore, we need a bound on $I_{\varepsilon}$ that does not depend on $\delta$. Indeed, applying a naive Cauchy Schwarz inequality, one can prove that  
\begin{align*}
    I_{\varepsilon}\leq \Lip(U^{\varepsilon})\Big(\frac{\varepsilon}{\delta}+\varepsilon\alpha\Big)\leq C\Big(\frac{\varepsilon}{\delta}+\varepsilon\alpha\Big).
\end{align*}
Optimizing w.r.t $\delta$ yields 
$$
 U^{\varepsilon}(\bar \mu)-U(\bar \nu)\leq C\alpha+ \sqrt{K(\varepsilon)\varepsilon},
$$
whose RHS  does not, a priori, vanish as $\varepsilon \to 0.$ We therefore need a sharper bound on $I_\varepsilon$ that does not depend on $\delta$, namely
\begin{equation*}
   I_\varepsilon= -\varepsilon \int_{\R^d}\big<\nabla \log \bar \mu(x), \frac {x- T^{\bar\nu}_{\bar\mu}(x)}{\alpha}\big>d\bar\mu(x)\leq d.
\end{equation*}
as proven in Lemma \ref{BoundTermineViscoso}, where the regularity assumptions therein are satisfied by the optimizer $\bar \mu\in \mathcal{P}_{2,\text{ac}}(\R^d)$.
This estimate can be regarded as the Wasserstein analogue of the finite-dimensional identity $\varepsilon\Delta_x \frac{|y-x|^2}{2\alpha}=\frac{d\varepsilon}{\alpha}$ which plays a crucial role in the finite dimensional proof in \cite{Bardi1997OptimalCA}. The inequality can be interpreted as a manifestation of the positive
curvature of the space $(\mathcal{P}_2(\mathbb{R}^d),W_2)$, as observed in
Remark~\eqref{rmk: Rimannianbound}.

The conclusion then follows as in the finite dimensional argument and, being the proof symmetric, we also obtain the converse inequality, completing the convergence result.

We then conclude the analysis showing that the one sided rate of convergence (from above) can be improved from $\sqrt{\varepsilon}$ to $\varepsilon$, provided $U^{\varepsilon}$ is a classical solution of \eqref{HJBId} such that $\Delta_x\mathcal{D}_{\mu}U^{\varepsilon}(t,\mu,x)\leq C$ uniformly. In particular, thanks to the analysis on geodesically concave functions, carried out in Section \ref{s: section3}, this rate is achieved by $C^2$ geodesically $\Lambda$-concave functions, at least for $d \geq 2$.

Finally, the above rates are optimal since they are optimal\footnote{w.r.t. the assumptions on the Hamiltonian. See \cite{cirant2025convergenceratesvanishingviscosity,chaintron2025optimalrateconvergencevanishing} for the optimal rate in the convex regime.} in the finite dimensional case that embeds in this mean field framework.

\end{itemize}

\subsection{Structure of the paper.}

In Section \ref{s:Set Up},  we fix the notation, we provide the main assumptions we are going to work with, and recall some important results from the literature that will be useful for our work. Moreover, we investigate some links between various notions of differentials. As a consequence, we will determine one meaningful differential to treat  both the equations \eqref{HJBId} and \eqref{HJBId}, in the viscosity sense.

In Section \ref{s: Hopf}, we prove in an intrinsic way a Hopf-Lax representation of the solution of \eqref{HJBDet}. 
In Section \ref{s: section3}, we derive estimates and lemmata that will be useful in Section \ref{s: vvl}, where we finally establish the vanishing viscosity result.

\section{The Wasserstein Space and the Related Subdifferential Calculus} \label{s:Set Up}

\subsection{The Wasserstein Space}
Our ambient space will be the space of probability measures on the $d$-dimensional Euclidean space $\R^d$ denoted as $\mathcal{P}(\R^d)$. The measure $\lambda_d$ will stand for the Lesbesgue measure in $\R^d$, and in the integration we will often use $dx$ rather than $\lambda_d$. 

Given $\mu \in \mathcal{P}(\R^d)$,  we will denote by 
\begin{align*}
\mathcal{M}_{2}(\mu) &:= \int_{\R^d}\vert x\vert^{2}d\mu, \\
    \mathcal{P}_{2}(\R^d)&:=\{ \mu \in \mathcal{P}(\R^d): \mathcal{M}_{2}(\mu)< \infty \},\\
    \mathcal{P}_{2,ac}(\R^d)&:= \{ \mu \in \mathcal{P}_2(\R^d): \mu << \lambda_d\}
\end{align*}
the second moment of the measure $\mu$ and the $L^2-$Wasserstein space, and the space of absolutely continuous measures with respect to $\lambda_d$, respectively.

The Wasserstein space $(\mathcal{P}_{2}(\R^d), W_2)$ is a complete and separable geodesic metric space when equipped with the Wasserstein distance $W_2: \mathcal{P}_2(\R^d)\times \mathcal{P}_2(\R^d)\to [0,\infty)$, defined for $\mu,\nu \in \mathcal{P}_2(\R^d)$ as 
\begin{equation}\label{def:wass}
    W_{2}(\mu, \nu):= \min_{\gamma \in \Gamma(\mu,\nu)} W^{\gamma}_2(\mu,\nu),
\end{equation}
where $$\Gamma(\mu, \nu): = \left\{ \gamma\in \mathcal{P}(\mathbb{R}^d \times \mathbb{R}^d) : {(\pi_1)}_{\#} \gamma = \mu, \, {(\pi_2)}_{\#} \gamma = \nu \right\}$$ is the set of couplings with marginals $\mu$ and $\nu$, $\pi_i:\R^d\times\R^d\to\R^d$ is the standard projection operator onto the $i$-th coordinate, $i\in \{1,2\}$, and, for a coupling {$\gamma \in \Gamma (\mu,\nu),$} 
\begin{equation}\label{eq: generalcost}
    W^{\gamma}_2(\mu,\nu):= \left(\int_{\R^{2d}} \vert x-y\vert ^{2} d\gamma(x,y)\right)^{\frac{1}{2}}
\end{equation}
is the transport cost from $\mu$ to $\nu$ associated to $\gamma$. For further details, see \cite[Chapter 4, Chapter 6]{OTVillani}.

We denote by $\Gamma_{0}(\mu, \nu)$ the set of optimal couplings between $\mu$ and $\nu$ for the distance $W_2$, i.e., any coupling that realizes the minimum cost in \eqref{def:wass}. 
In general, we will say that a plan $\gamma \in \mathcal{P}(\R^{2d})$ is optimal if $\gamma \in \Gamma_{0}((\pi_1)_\# \gamma, (\pi_2)_\# \gamma)$, i.e. $\gamma$ is optimal w.r.t. its marginals.

Moreover, for ${\gamma} \in \mathcal{P}(\R^{2d}), \nu \in \mathcal{P}(\R^d)$, the sets 
$$\Gamma({\gamma},\nu):=\{\theta \in \mathcal{P}(\R^{3d}): (\pi_1,\pi_2)_{\#}\theta= \gamma, \ (\pi_1,\pi_3)_{\#}\theta \in \Gamma({(\pi_1)}_{\#}\gamma, \nu)\},$$  
$$\Gamma_0({\gamma},\nu):=\{\theta \in \mathcal{P}(\R^{3d}): (\pi_1,\pi_2)_{\#}\theta= \gamma, \ (\pi_1,\pi_3)_{\#}\theta \in \Gamma_0({(\pi_1)}_{\#}\gamma, \nu)\}$$
will denote the set of (resp. optimal) couplings between the plan $\gamma$ and the measure $\nu$. Here $\pi_i:\R^{3d}\ \to\R^d$ is the standard projection operator onto the $i$-th coordinate, $i\in \{1,2,3\}$, coherently with  the notation used for the projection from $\R^{2d}$ to $\R^d$.

To any $\gamma \in \mathcal P(\R^{2d})$ we associate its disintegration w.r.t. the first marginal $\mu:=(\pi_1)_\# \gamma$:  i.e. the $\mu$-a.e. uniquely determined measure $\gamma_{x}(dy)\in \mathcal P(\R^d)$ s.t. $\gamma=\gamma_{x}(dy)\mu(dx)$. 
The disintegration easily generalizes to the second marginal and to every marginal of measures $\theta\in \mathcal P(\R^{3d})$. See \cite[ Theorem 5.3.1]{AGS}.

Let $\mathcal{F} \colon  \mathcal{P}_{2}(\R^d) \to \R$ be a function. We call 
\begin{equation*}
    \Lip(\mathcal{F}; W_{2}):=\inf_{\mu\not=\nu} \frac{ \vert \mathcal{F}(\mu) -\mathcal{F}(\nu) \vert}{W_{2}(\mu, \nu)}
\end{equation*}
the Lipschitz constant of $\mathcal{F}$ w.r.t. the $W_2$ distance.

For future use, we mention that for any $\mu,\nu \in \mathcal P_2(\R^d)$, as a simple consequence of the Cauchy Schwarz inequality we have the following inequalities 
\begin{equation*} 
\vert \mathcal{M}_2( \mu)- \mathcal{M}_2( \nu) \vert \leq  W_2(\mu,\nu) \left( \mathcal{M}_2( \mu)^{\frac{1}{2}}+\mathcal{M}_2( \nu)^{\frac{1}{2}}\right).
\end{equation*}

In particular, we have the following locally Lipschitz bound for second moments
\begin{equation} \label{rk: momentsdist}
    \vert \mathcal{M}_2(\mu)- \mathcal{M}_2( \nu) \vert \leq  \frac{1}{2} W_2(\mu,\nu)\left(2+ \mathcal{M}_2( \mu)+\mathcal{M}_2(\nu)\right).
\end{equation}
Indeed, given $ \sigma \in \Gamma_0(\mu, \nu)$, by standard algebra, we have

\begin{align*}
 \mathcal{M}_2( \mu)- \mathcal{M}_2( \nu)&=\int_{\R^d\times\R^d}(x-y)\cdot(x+y)d \sigma(x,y)\\
 &\leq \left(\int_{\R^d\times\R^d}\vert x-y\vert ^2 d  \sigma(x,y)\left)^{\frac{1}{2}}\right(\int_{\R^d\times\R^d}\vert x+y\vert ^2 d \sigma(x,y)\right)^{\frac{1}{2}}\\
 &\leq W
 _2(\mu,\nu)\left((\int_{\R^d}\vert x\vert ^2 d\mu(x))^{\frac{1}{2}}+(\int_{\R^d}\vert y\vert ^2 d\nu(y))^{\frac{1}{2}}\right).
\end{align*}

Note that the previous notions can be easily generalized to $\mathcal M_p (\cdot), \mathcal P_p(\R^d), W_p(\cdot,\cdot), \Lip(\cdot; W_{p})$ for $p\geq 1$.

\subsection{Tangent spaces at a measure}

We survey here the definitions of various notions of tangent space at a measure, as described in \cite[Appendix]{AGS}.
Thanks to this description, we can unify the notions of viscosity solution in the first order case and in the semilinear case with idiosyncratic noise.

\begin{defn}\label{d: TanM}
Let $\mu \in \mathcal{P}_2(\R^d)$. The \emph{(Monge) tangent space at $\mu$} is defined as the set
\begin{equation*}
    \Tan^{M}_{\mu}\sP_2(\R^d) := \overline{\left\{ \varepsilon(r - \mbox{Id}) \colon r:\R^d\to\R^d \mbox{ Borel measurable}, (\mbox{Id}\times r)_{\#}\mu \text{ is optimal, for }\varepsilon > 0 \right\}}^{L^2_\mu(\R^d)},
\end{equation*}
where $L^2_\mu(\R^d)$ is the $L^2$ space of $\mu$-measurable $\R^d$-valued maps.
\end{defn}

Given $\mu\in \mathcal P_2(\R^d)$, let us call 
$$
    \mathcal{P}_{2,\mu}(\R^{2d}):=\left\{\gamma \in \mathcal{P}_2{(\R^{2d})}, (\pi_1)_{\#}\gamma=\mu \right\},
$$
the set of plans in $\mathcal{P}_2{(\R^{2d})}$ having $\mu$ as first marginal, and for $\gamma,\xi \in \mathcal{P}_{2,\mu}({\R^{2d}})$
$$
    \Gamma_{\mu}(\gamma,\xi):=\{\theta \in \mathcal P (\R^{3d}), (\pi_{1,2})_{\#}\theta = \gamma, (\pi_{1,3})_{\#}\theta=\xi\},
$$
 the set of plans in $\mathcal{P}{(\R^{3d})}$, couplings in $\mathcal P_{2,\mu}(\R^{2d})$. Here, $\pi_{1,j}:\R^d\times\R^d\times\R^d\to \R^d\times\R^d$ is the projection on the first and the $j$-th coordinate $\pi_{1,j}:=(\pi_1, \pi_j)$, $j=2,3$. To shorten the notation we will write $\gamma^{\mu}$ instead of $\gamma$, when the latter belongs to $\mathcal{P}_{2,\mu}(\R^{2d})$. 
 
The set $\mathcal{P}_{2,\mu}(\R^{2d})$ is metrized by the distance $W_\mu \colon \mathcal{P}_{2,\mu}(\R^{2d}) \times \mathcal{P}_{2,\mu}(\R^{2d}) \to [0,\infty)$ defined  in \cite[Equation ~(12.4.10)]{AGS},  as
\begin{equation}\label{DistanceTangentBundle}
W^2_{\mu}(\gamma^\mu_{1},\gamma^\mu_{2}):=\min_{\theta \in \Gamma_{\mu}(\gamma_1^\mu,\gamma_2^\mu)} \left\{\int_{\R^{3d}} \vert y-z\vert^2  d\theta(x,y,z) \right\}.
\end{equation}

\begin{defn}\label{d: TanK}
Let $\mu \in \mathcal{P}_2(\R^d)$. The \emph{(Kantorovich) geometric tangent space at $\mu$} is defined as:
\begin{align*}
    \Tan^{K}_{\mu}\sP_2(\R^d) := \overline{
    \left\{ \gamma \in \mathcal{P}_{2,\mu}(\R^{2d}) : (\pi_1,\pi_1 + \varepsilon \pi_2)_{\#}\gamma \text{ is optimal, for some } \varepsilon > 0 \right\} }^{W_\mu(\cdot, \cdot)}
\end{align*}
\end{defn}

In $\mathcal{P}_{2,\mu}(\R^{2d})$, it is naturally defined an \emph{exponential map}:
\begin{align*}\label{d:exp}
    \exp_{\mu}: &\ \mathcal{P}_{2,\mu}(\R^{2d})\to \mathcal{P}_2(\R^d) \nonumber \\
    &\ \gamma\longmapsto (\pi_1+\pi_2)_{\#}\gamma.
\end{align*}

We remark that, as in \cite[Definition ~4.18]{GigliPhd}, we prefer to define the exponential on the whole $\mathcal{P}_{2,\mu}(\R^{2d})$ rather than on its closed subset $\Tan^K_{\mu}\mathcal{P}_2(\R^d)$ for reasons that will become clear later in our discussion about differentials. Moreover, let us note that $\exp_\mu$ is surjective, i.e., more precisely, for all $\nu \in \mathcal{P}_2(\R^d)$ there exists an optimal velocity $\gamma \in \Tan^K_{\mu}\mathcal{P}_2(\R^d)$
s.t. $\exp_\mu(\gamma)=\nu$.

Following \cite[Definition ~4.13]{GigliPhd}, given $\mu\in \mathcal P_2(\R^d)$, we can also introduce a notion of ``scalar product" on  $\mathcal{P}_{2,\mu}(\R^{2d})$ as
\begin{equation}\label{scalar}
    \big<\gamma_1^\mu,
    \gamma_2^\mu\big>_{\mu}:=\max_{\theta \in \Gamma_{\mu}(\gamma_1^\mu,
    \gamma_2^\mu)} \big<\gamma_1^\mu,
    \gamma_2^\mu\big>_{\theta,\mu},
\end{equation}
where  for $\theta \in \Gamma_{\mu}(\gamma_1^\mu,    \gamma_2^\mu) $,  $\big<\gamma_1^\mu,\gamma_2^\mu\big>_{\theta,\mu}:=\int_{\R^{3d}}\big<y,z\big>d\theta(x,y,z) $. 

Note that the extrema of \eqref{scalar} and \eqref{DistanceTangentBundle} coincide because of the identity $|y-z|^2=|y|^2+|z|^2 -2\big<y,z\big>, \quad \forall y,z \in \R^d$. Moreover, whenever either $\gamma_1^\mu$ or $\gamma_2^\mu$ is induced by a map, say $\gamma_2^\mu=(\Id \times p)_{\#}\mu$ with $p \in L^2_{\mu}(\R^d)$, being $\Gamma_{\mu}(\gamma_1^\mu,\gamma_2^\mu)$ a singleton, we have
$$\big<\gamma_1^\mu,\gamma_2^\mu\big>_{\mu}=\int\big<y,p(x)\big>d\gamma_1^\mu(x,y).
$$  
As a consequence one can also introduce a \emph{Wasserstein norm} in $ \mathcal{P}_{2,\mu}(\R^{2d})$,  defined for $\xi \in \mathcal{P}_{2,\mu}(\R^{2d})$ as 
\begin{equation}\label{e: normW}
\|\xi\|_{\mu}:=\sqrt{\big<\xi,\xi\big>_{\mu}}=W_{\mu}(\xi,(\Id \times 0)_{\#}\mu).
\end{equation}
In particular, the norm is continuous w.r.t. the convergence induced by $W_{\mu}$. Moreover, observe that 
\begin{equation*}
    W_2(\mu,\exp_{\mu}(\xi))\leq \|\xi\|_{\mu}, \quad \forall \xi \in \mathcal{P}_{2,\mu}(\R^{2d})
\end{equation*}
where the equality holds if $\xi \in \Gamma_{0}(\mu,\exp_{\mu}(\xi)).$

By \cite[Proposition ~4.21]{GigliPhd}, we have  that, for $\gamma_1,\gamma_2,\xi \in \mathcal{P}_{2,\mu}(\R^{2d})$, 
\begin{equation*}
    |\big<\gamma_1,\xi\big>_\mu- \big< \gamma_2,\xi\big>_\mu|\leq W_{\mu}(\gamma_1,\gamma_2)\|\xi\|_\mu.
\end{equation*}
In other words, the scalar product is  Lipschitz continuous w.r.t. $W_\mu$. 
We also remark, see \cite[Remark ~4.28]{GigliPhd}, that in general 
\begin{equation}\label{e: max> min}
\max_{\theta \in \Gamma_{\mu,\nu}(\gamma,\xi)} \big<\gamma,\xi\big>_{\theta,\mu}\geq \min_{\theta \in \Gamma_{\mu,\nu}(\gamma,\xi)} \big<\gamma,\xi\big>_{\theta,\mu},
\end{equation}
with the inequality possibly strict.
However, the equality holds if either $\gamma$ or $\xi$ is induced by a map. 
Lastly,  a multiplication operation can be defined on $\mathcal{P}_{2,\mu}(\R^{2d})$, see \cite[Proposition ~4.22 $\&$ Proposition ~4.24]{GigliPhd}: given $\gamma \in \mathcal{P}_{2,\mu}(\R^{2d})$
\begin{equation*}
    \R \ni\lambda \mapsto \lambda \cdot \gamma:=(\pi_1,\lambda \pi_2)_{\#}\gamma. 
\end{equation*}
If $\gamma \in \Tan^{K}_{\mu}\mathcal{P}_2(\R^d)$, it is immediate to see that $\lambda \cdot \gamma \in \Tan_{\mu}^K{\mathcal P}_{2}(\R^{d})$, for every $\lambda \geq 0$. This invariance is still true for $\lambda < 0$, see \cite[Proposition ~4.29]{GigliPhd}.  {To lighten the notation we will often write $-\gamma$ for $-1\cdot \gamma$.} 
We also remark that the product $\big<\cdot,\cdot\big>_{\mu}$ enjoys a one sided Cauchy Schwarz inequality: Indeed, for every $\gamma^1,\gamma^2 \in \mathcal{P}_{2,\mu}(\R^{2d})$
\begin{align*}
    \big<\gamma^1,\gamma^2\big>_{\mu}&=\max_{\theta \in \Gamma_{\mu}(\gamma^{1},\gamma^{2})}\int \big<y,z\big>d\theta(x,y,z)\\
    &\leq \max_{\theta \in \Gamma_{\mu}(\gamma^{1},\gamma^{2})}\Big(\int |y|^2d\theta(x,y,z)\Big)^{\frac{1}{2}}\Big(\int |z|^2d\theta(x,y,z)\Big)^{\frac{1}{2}}\\
    &= \Big(\int |y|^2d\gamma^{1}(x,y)\Big)^{\frac{1}{2}}\Big(\int |z|^2d\gamma^{2}(x,z)\Big)^{\frac{1}{2}}\\
    &=\|\gamma^{1}\|_{\mu}\|\gamma^{2}\|_{\mu}.
\end{align*}

Moreover, the equality holds iff there exists $\theta \in \Gamma_{\mu}(\gamma^{1},\gamma^{2})$ optimizer for $\big<\cdot,\cdot\big>_{\mu}$ s.t. $\big<y,z\big>=|y||z|$ holds $\theta$-a.e., i.e. $\gamma^1=\lambda \cdot\gamma^2$, with $\lambda\geq 0$.

As in \cite[Proposition ~4.30]{GigliPhd}, we define a projection operator on $\Tan^{K}_{\mu}\mathcal{P}_2(\R^{2d})$ as:
\begin{equation*}
    \text{P}:\mathcal{P}_{2,\mu}(\R^{2d}) \to \Tan^{K}_{\mu}\mathcal{P}_2(\R^{d}),
\end{equation*}
that at each element $\gamma$ associates the unique plan $\text{P}(\gamma)\in \Tan^K_{\mu}\mathcal{P}_2(\R^d)$ that minimizes the distance between $\gamma$ and the geometric tangent space $\Tan^K_{\mu}\mathcal{P}_2(\R^d)$. In particular, see  \cite[Corollary ~4.34]{GigliPhd}, 
\begin{equation}\label{projectionOperator}
    \big<\text{P}(\gamma),\eta\big>_{\mu}=\big<\gamma,\eta\big>_{\mu} 
    \quad \forall \eta \in \Tan^{K}_{\mu}\mathcal{P}_2(\R^d) \quad \text{and} \quad \|\text{P}(\gamma)\|_{\mu}\leq \|\gamma\|_{\mu}, \quad \forall \gamma \in \mathcal{P}_{2,\mu}(\R^{2d}).
\end{equation}

In order to understand the link between the two different notions of tangent space, let us consider the barycentric projection (see \cite[ Definition ~5.4.2.]{AGS}). For a fixed $\mu\in \mathcal P_2(\R^d)$, the barycentric projection map is defined as 
\begin{align*}
    \text{Pr}_{\mu} \colon \mathcal{P}_{2,\mu}(\R^{2d}) &\to L^2_\mu(\R^d),\\
    \gamma &\mapsto \bar{\gamma},
\end{align*}
where, given $\gamma(dx, dy) = d\gamma_x(y) d\mu(x)$, the disintegration of $\gamma$ with respect to its  first marginal $\mu=(\pi_1)_\# \gamma$,  the image $\bar{\gamma}:\R^d\to \R^d$ is defined as :
\begin{equation*}
   \bar{\gamma}(x): = \int_{\R^d} y \, d\gamma_x(y).
\end{equation*}
Note that Jensen inequality ensures the wellposedness of the above map, indeed $\bar\gamma\in L^2_\mu(\R^d)$.
To be more general, we  define the projection map\footnote{In a probabilistic framework, the barycentric projection of the law of a random vector $(X,Y)$ corresponds to the law of the conditional expectation $\mathbb{E}[Y \vert X]$.} on the whole $\mathcal{P}_2(\R^{2d})$, as 
\begin{equation}\label{Pr}
    \text{Pr}(\gamma):=\text{Pr}_{(\pi_1)_{\#}\gamma}(\gamma), \quad \forall \gamma \in \mathcal{P}_2(\R^{2d}).
\end{equation}

It turns out that, for any $\mu\in \mathcal P_2(\R^d)$, the Monge tangent space can be seen as the image of the restriction of the barycentric projection to the geometric tangent space. 
Moreover, it is not difficult to show that $\text{Pr}$ is onto and a $1-$contraction (see \cite[Definition~12.4.3]{AGS},  and the subsequent discussion). Namely
\begin{align}
    \text{Pr}(\Tan^{K}_{\mu}\sP_2(\R^d))&=\Tan^{M}_{\mu}\sP_2(\R^d), \label {e: proj} \\
\|\text{Pr}(\gamma_1)-\text{Pr}(\gamma_2)\|_{L^2_{\mu}(\R^d)}&\leq W_\mu(\gamma_1,\gamma_2) \quad \forall \gamma_1,\gamma_2\in \mathcal{P}_{2,\mu}(\R^{2d}). \label{e: Pr}
\end{align}

When $\mu$ is absolutely continuous, we can be more precise: this projection is a one-to-one isometry w.r.t. the distances previously introduced, as shown in \cite[Theorem ~12.4.4]{AGS}. Thus
\begin{equation}\label{isometryTangent}
    (\Tan^{K}_{\mu}\mathcal P_2(\R^d),W_\mu)\simeq_{\text{Pr}}(\Tan^{M}_{\mu}\mathcal P_2(\R^d),\|\cdot\|_{L^2_\mu}).
\end{equation}

Let us note that, using the characterization of optimal maps for absolutely continuous measures, we also have the following alternative description for the tangent space (see  \cite[Definition ~8.4.1 $\&$ Theorem ~8.5.1]{AGS}):
\begin{equation}\label{Tangentspace}
    \Tan^{M}_{\mu}\mathcal P_2(\R^d) = \overline{\left\{ \nabla \phi \colon \phi \in C^{\infty}_c(\R^d) \right\}}^{L^2_\mu(\R^d)}.
\end{equation}

In particular, at an absolutely continuous measure $\mu$ the scalar product \eqref{scalar} is the standard scalar product of $L_{\mu}^2(\R^d)$ restricted to $\Tan^{M}_{\mu}\mathcal P_2(\R^d)$.

\subsection{Subdifferential calculus in the Wasserstein Space.}
\subsubsection{First variation and Flat structure}

We now briefly recap the principal notions of differential calculus in Wasserstein spaces. For more details, see \cite[Section ~10.3]{AGS} and \cite[Chapter~5]{CarmonaDelarue}.
At the end of the discussion, we will select a unique differential that can be used  for equations  \eqref{HJBDet} and \eqref{HJBId}.\\

We begin with the definition of the flat derivative, noting that the Wasserstein space can be seen as a convex subset of the vector space of finite Borel measures $\mathcal{M}(\R^d)$. 
We adopt a more general definition, inspired by the notions of first derivative on $\mathcal{P}_2(\R^d)$ introduced in  \cite{PierreCaNotes} and \cite{santambrogio2015optimal}. The growth assumption is taken from \cite[Chapter ~5]{CarmonaDelarue}. 
\begin{defn}
Let \( U : \mathcal{P}_2(\mathbb{R}^d) \to \mathbb{R} \). 
We say that \( U \) is \emph{flat differentiable} at $\mu$, if there exists  a Borel measurable map with quadratic growth denoted 
\[
\mathcal{D}_{\mu}U(\mu,\cdot) : \mathbb{R}^d \to \mathbb{R},
\]
s.t., $\forall  \nu \in \mathcal{P}_2(\R^d)$, it holds
\[
\lim_{h \to 0^+}\frac{U( (1-h)\mu + h\nu) - U(\mu)}{h} =  \int_{\mathbb{R}^d} \mathcal{D}_{\mu}U(\mu,x) (\nu - \mu)(dx)\, .
\]

We say that $U:\mathcal P_2(\mathbb R^d)\to\mathbb R$ is $C^1$
(resp.\ \emph{fully} $C^1$) if there exists $\mathcal D_\mu U : \mathcal{P}_2(\R^d)\times \R^d \to \R$ that is jointly continuous with uniform quadratic growth (resp. if $U$ is $C^1$ and there exists $\nabla \mathcal D_\mu U: \mathcal{P}_2(\R^d)\times \R^d \to \R^d$ that is jointly continuous and with uniform linear growth). 

We say that $U:\mathcal P_2(\mathbb R^d)\to\mathbb R$ is \emph{partially} $C^2$ if  $U$ is fully $C^1$ and  there exists $\nabla^2\mathcal D_\mu U : \mathcal{P}_2(\R^d)\times \R^d \to M_{d\times d}(\R)$ that is jointly continuous and uniformly bounded.

\end{defn}

Since the convex interpolation is restricted to be in $\mathcal{P}_2(\R^d)$ and $\int \text{cost} (\nu-\mu) dy =0$, the flat derivative $\mathcal{D}_{\mu}U$ is defined up to a constant   depending only on $\mu$. We select this constant s.t. $\int_{\R^d}\mathcal{D}_{\mu}U(\mu,x)d\mu(x)=0.$ Due to this choice of constant, we have the following representation formula: if $U$ is flat differentiable at $\mu$, then
\begin{equation}\label{e:FlatderivativeRepresentiation}
    \mathcal{D}_{\mu}U(\mu,y)=  \lim_{h \to 0^+} \frac{1}{h} \left( U\left((1-h)\mu + h\delta_y\right) - U(\mu) \right), \quad \forall y \in \R^d. 
\end{equation}

\subsubsection{Wasserstein differential calculus}
Various notions of differentiability have been introduced in the literature. Let us start with two classical notions that are given with two different points of view.  The first one sees the elements of the differential as plans while the second one as vector valued functions. The two notions can be related via the barycentric projection map introducing an intermediate notion of differential.\\

In the following, we will provide only the definition of subdifferential of a proper and lower semicontinuous function since the superdifferential  of a proper and upper semicontinuous function  $\phi \colon \mathcal{P}_2(\R^d) \to [-\infty, \infty )$ can be obtained by setting $\partial^{+}_{\cdot}\phi(\mu):=\big\{-1\cdot\gamma: \gamma \in \partial_{\cdot}^{-}(-\phi)(\mu)\big\}$ in the definitions \ref{d: classicdiff} \ref{d: inter_diff}, \ref{d: SDiff1} and \ref{d: SDiff}.  As already remarked, this operation leaves invariant the geometric tangent space. Whenever $\partial^{-}_\cdot (-\phi)(\mu) \subset L^2_\mu(\R^d)$, as in Definitions \ref{d: classicdiffL} and \ref{d: sDiffL}, we simply consider $\partial^{+}_\cdot \phi(\mu)=-\partial^{-}_\cdot (-\phi)(\mu)$. We denote by $\text{Dom} (\phi):= \Big\{\mu \in \mathcal{P}_2(\R^d): \phi(\mu)<\infty\Big\}$ the effective domain of $\phi$.

In what follows, we introduce the various notions of differential within the formalism established above. For each definition, we explicitly indicate the corresponding reference.

\begin{defn}[See  \cite{GigliPhd}, Definition  5.14] \label{d: classicdiff} 
Let $\phi \colon \mathcal{P}_2(\R^d) \to (-\infty, +\infty ]$ be proper and lower semicontinuous. Let $\mu\in \text{Dom} (\phi)$. We say that $\gamma \in \mathcal{P}_{2, \mu}(\R^{2d})$ 
belongs to the Kantorovich subdifferential of $\phi$ at $\mu$, denoted $\partial^{-}_{K}\phi(\mu)$, if for all $\xi \in \Tan^K_{\mu}\mathcal{P}_2(\R^{d})$
    \begin{equation*}
    \phi(\exp_{\mu}(\xi)) - \phi(\mu) \geq -\big<-\gamma,\xi\big>_{\mu} + o(\|\xi\|_{\mu}).
    \end{equation*}
\end{defn}

Note that in \cite[Definition~5.14]{GigliPhd}, the author additionally requires that $\partial^{-}_K \phi(\mu) \subset \Tan^K_{\mu} \mathcal{P}_2(\R^d)$.
In the present work, we choose not to impose this restriction at the level of the definition, in order to allow for a more refined comparison with other notions.

\begin{defn}[See \cite{GANGBO2019119}, Definition 3.1]
  \label{d: classicdiffL}
Let $\phi \colon \mathcal{P}_2(\R^d) \to (-\infty, +\infty ]$ be proper and lower semicontinuous. Let $\mu\in \text{Dom} (\phi)$. We say that $p \in L^2_\mu(\R^d)$ 
belongs to the Monge subdifferential of $\phi$ at $\mu$, denoted $\partial^{-}_{M}\phi(\mu)$, if for all $\xi \in \Tan^K_{\mu}\mathcal{P}_2(\R^d)$ 
    \begin{equation*}
    \phi(\exp_{\mu}(\xi)) - \phi(\mu) \geq  \int_{\R^{2d}} \langle p(x), z \rangle \, d\xi(x,z) + o(\|\xi\|_{\mu}).
\end{equation*}
\end{defn}

In order to formalize the link between the previous two classical notions, let us introduce the following  intermediate definition, given in the spirit of Definition \ref{d: SDiff} below, introduced in \cite{BSOT}.

\begin{defn}[Intermediate Definition]  \label{d: inter_diff}
Let $\phi \colon \mathcal{P}_2(\R^d) \to (-\infty, +\infty ]$ be proper and lower semicontinuous. Let $\mu\in \text{Dom} (\phi)$.  We say that $\gamma \in \mathcal{P}_{2,\mu}(\R^{2d})$ 
belongs to the intermediate Kantorovich subdifferential of $\phi$ at $\mu$, denoted $\partial^{-}_{I}\phi(\mu)$, if for all $\xi \in \Tan^K_{\mu}\mathcal{P}_2(\R^d)$ 
    \begin{equation*}
    \phi(\exp_{\mu}(\xi)) - \phi(\mu) \geq  \int_{\R^{2d}} \big< y, z \big>  \gamma_x(dy)d\xi(x,z) + o(\|\xi\|_{\mu}),
\end{equation*}
    where $\gamma_{x}$ is the disintegration of $\gamma$ with respect to its first marginal $\mu$, i.e. $\gamma=\gamma_{x}(dy)\mu(dx)$. 

\end{defn}

Let us note that
$$\partial^{-}_{I}\phi(\mu)\subseteq \partial^{-}_{K}\phi(\mu)\quad \text{and} \quad \text{Pr}(\partial^{-}_{I}\phi(\mu))= \partial^{-}_{M}\phi(\mu).$$
Moreover, intersecting the subdifferential with the geometric tangent space, due to the properties of the barycentric projection map on the geometric tangent space, see equation \eqref{e: proj}, we have $$\text{Pr}(\partial^{-}_{I}\phi(\mu)\cap\Tan^{K}_{\mu}\sP_2(\R^d))= \partial^{-}_{M}\phi(\mu)\cap\Tan^{M}_{\mu}\sP_2(\R^d).$$

Suppose $\mu$ is absolutely continuous, then $\Tan^K_{\mu}\mathcal{P}(\R^d)$ can be isometrically identified  with $\Tan^M_{\mu}\mathcal{P}_2(\R^d)\subset L^2_{\mu}(\R^d)$. Therefore, the three inequalities above can be, equivalently, tested on maps and read: for all $\xi\in L^2_\mu(\R^d)$
\begin{equation*}
    \phi(\exp_{\mu}((\Id \times \xi )_{\#}\mu)) - \phi(\mu) \geq \int_{\R^{d}} \big< \Pr( \gamma)(x),\xi (x) \big>d\mu(x)  + o(\|\xi\|_{L^2_\mu}),
\end{equation*}
since in this case $\big<(\text{Id}\times \xi)_\#{\mu},\gamma\big>_{\mu}=\big<\xi,\Pr(\gamma)\big>_{L^2_{\mu}(\R^d)}.$
Moreover, being the barycentric projection a one-to-one isometry,  we have $$\partial^{-}_{K}\phi(\mu)\cap\Tan^{K}_{\mu}\sP_2(\R^d) \simeq_\text{Pr}\partial^{-}_{M}\phi(\mu)\cap\Tan^{M}_{\mu}\sP_2(\R^d),$$
thus establishing a strong  connection between the two differentials  in the case of absolutely continuous measures.

Note that the set $\partial^{-}_{M}\phi(\mu)\cap\Tan^{M}_{\mu}\sP_2(\R^d)$ is the subdifferential set defined in \cite{DS}. \\

The unique element in $\partial^{+}_{K}\phi(\mu)\cap \partial^{-}_{K}\phi(\mu) \cap \Tan^K_{\mu}\mathcal{P}_2(\R^d)$, if any, is called the Wasserstein differential of $\phi$, see \cite[Definition ~5.14]{GigliPhd}. In this case, we say that $\phi$ is differentiable at $\mu$ and $\partial_{\mu}\phi$ is the Wasserstein differential at $\mu$.

A stronger notion of subdifferential can be obtained by requiring, in Definitions \ref{d: classicdiff},\ref{d: classicdiffL} and \ref{d: inter_diff}, that the corresponding inequalities hold for a larger class of perturbations, e.g. without assuming the measure vector field $\xi$ to lie in the tangent space. This restriction will select more stable elements in the subdifferentials, and it has already been used in the gradient flow theory developed in\cite{AGS} to treat the Euler equation of the minima of the Yosida approximation of a given functional. We also remark that this strongest notion is linked with the so-called Hilbertian lift. See \cite[Remark ~2.5 $\&$ Proposition ~2.6]{BSOT} and the introduction in \cite{GANGBO2019119}.

\begin{defn}[See \cite{AGS}, Definition 10.3.1] Strong Fr\'echet Subdifferential]\label{d: SDiff1}
Let $\phi  \colon \mathcal{P}_2(\R^d) \to (-\infty, +\infty ]$ be proper and lower semicontinuous. Let $\mu\in \text{Dom} (\phi)$. We say that $\gamma \in \mathcal P_{2,\mu}(\R^{2d})$ belongs to the strong subdifferential of $\phi$ at $\mu$, denoted $\partial^{-}_{S}\phi(\mu)$, if for all $\xi\in \mathcal{P}_{2,\mu}(\R^{2d})$ and $\theta \in \Gamma_{\mu}(\gamma, \xi)$,
    \begin{equation*}
        \phi(\exp_{\mu}(\xi)) - \phi(\mu) \geq \big<\xi,\gamma\big>_{\theta,\mu} + o(\|\xi\|_{\mu}).
    \end{equation*}
\end{defn}

\begin{defn} \label{d: sDiffL}
Let $\phi \colon \mathcal{P}_2(\R^d) \to (-\infty, +\infty ]$ be proper and lower semicontinuous. Let $\mu\in \text{Dom} (\phi)$. We say that $p \in L^2_\mu(\R^d)$ 
belongs to the (strong-Monge) subdifferential of $\phi$ at $\mu$, denoted $\partial^{-}_{S,M}\phi(\mu)$, if for all $\xi \in \mathcal{P}_{2,\mu}(\R^{2d})$  
    \begin{equation*}
    \phi(\exp_{\mu}(\xi)) - \phi(\mu) \geq \int_{\R^{2d}} \big< p(x), z \big> d\xi(x,z) + o(\|\xi\|_{\mu}).
\end{equation*}

\end{defn}

\begin{defn}[See \cite{BSOT}, Definition 2.2] \label{d: SDiff}
Let $\phi  \colon \mathcal{P}_2(\R^d) \to (-\infty, +\infty ]$ be proper and lower semicontinuous. Let $\mu\in \text{Dom} (\phi)$. We say that $\gamma \in \mathcal P_{2,\mu}(\R^{2d})$ belongs to the map strong subdifferential of $\phi$ at $\mu$, denoted $\partial^{-}_{S,I}\phi(\mu)$, if for all $\xi \in \mathcal{P}_{2,\mu}(\R^{2d})$
    \begin{equation*}\label{StrongDiffBertucci}
        \phi(\exp_{\mu}(\xi)) - \phi(\mu) \geq \int_{\R^{3d}} \big< y, z \big> \gamma_{x}(dy)d\xi(x,z) + o(\|\xi\|_{\mu}),
    \end{equation*}
    where $\gamma_{x}$ is the disintegration of $\gamma$ with respect to its first marginal $\mu$, i.e. $\gamma=\gamma_{x}(dy)\mu(dx)$. 

\end{defn}

The above definition was introduced in \cite{BSOT} to prove the well-posedness of viscosity solutions of Hamilton-Jacobi equations in $\mathcal{P}(\mathbb{T}^d)$. 

\begin{rmk}
    Let us note that w.r.t.  \cite[Definition ~2.2]{BSOT}, we preferred here to describe the elements of the subdifferential as plans instead of using the equivalent identification of $\gamma$ with $\gamma_{x}$ $\mu-$a.e. determined by $\gamma(dx,dy)=\gamma_{x}(dy)\mu(dx)$. We do the same for the notion of viscosity solutions adopted in the first-order case.
\end{rmk}

As for the previous definitions, let us note that
$$\partial^{-}_{S}\phi(\mu)\subseteq \partial^{-}_{S,I}\phi(\mu), \quad \quad \text{Pr}(\partial^{-}_{S,I}\phi(\mu))= \partial^{-}_{S,M}\phi(\mu)  \quad \text{ and } \quad  \text{Pr}(\partial^{-}_{S}\phi(\mu))\subseteq \partial_{S,M}^{-}\phi(\mu)$$
Moreover,  we have
    $$\text{Pr} (\partial ^-_{S,I} \phi(\mu)\cap\Tan^{K}_{\mu}\sP_2(\R^d))= \partial ^-_{S,M} \phi(\mu) \cap \Tan^{M}_{\mu}\sP_2(\R^d),$$
    while the isometries $$\partial^{-}_{S,I}\phi(\mu)\cap\Tan^{K}_{\mu}\sP_2(\R^d) \simeq_\text{Pr} \partial^{-}_{S,M}\phi(\mu)\cap\Tan^{M}_{\mu}\sP_2(\R^d)$$ 
 holds for absolutely continuous $\mu$. 
\begin{rmk}
    Let us note that 
    $$ \partial ^-_S \phi(\mu)\subseteq \partial ^-_{S, I} \phi(\mu)\subseteq \partial ^-_{I} \phi(\mu)\subseteq \partial ^-_K \phi(\mu), $$
    and
    $$ \partial ^-_{S,M} \phi(\mu)\subseteq  \partial ^-_{M} \phi(\mu).  $$
   Obviously, the same inclusions holds, with suitable modifications, for the superdifferentials.
\end{rmk}
\begin{lem}\label{Strong=Opt}(Correspondent of \cite[Theorem ~3.14]{GANGBO2019119} in our setting)
Let $\phi: \mathcal{P}_2(\R^d)\to (-\infty,+ \infty]$ be a proper lower semicontinuous functional, and $\mu \in \text{Dom} (\phi)$. Then,
\begin{align*}
    \partial_{S,M}^{-}\phi(\mu)\cap \Tan^{M}_{\mu}\mathcal{P}_2(\R^d)&=\partial^{-}_{M}\phi(\mu)\cap \Tan^{M}_{\mu}\mathcal{P}_2(\R^d).
\end{align*}
Moreover, when $\mu$ is absolutely continuous w.r.t. the Lesbsegue measure
\begin{equation}\label{eq:last}
    \partial^{-}_S\phi(\mu)\cap \Tan^K_{\mu}\mathcal{P}_2(\R^d)\simeq_\mathrm{Pr}\partial^{-}_{M}\phi(\mu)\cap \Tan^{M}_{\mu}\mathcal{P}_2(\R^d).
\end{equation}
\end{lem}
\begin{proof}
In view of the inclusion $\partial ^-_{S,M} \phi(\mu)\subseteq  \partial ^-_{M} \phi(\mu)$, it is sufficient to prove only the $\supseteq$ part. The proof is an adaptation of \cite[ Lemma ~3.2]{Erbar2010}.

Suppose the the inclusion does not hold: There exists $p\in \partial^{-}_{M}\phi(\mu)\cap \Tan^{M}_{\mu}\mathcal{P}_2(\R^d), \, \delta>0$ and a sequence $(\xi_n)_n \subset \mathcal{P}_{2,\mu}(\R^d)$ with $\|\xi_n\|_\mu=\varepsilon_n \to 0$ s.t. 
\begin{equation*}
    \phi(\exp_{\mu}(\xi_n))-\phi(\mu)-\int_{\R^{2d}}\big<p(x),z\big>d\xi_n(x,z)\leq -\delta\varepsilon_n.
\end{equation*}

Choose $(\tilde \xi_n)_n \subset \Tan^K_{\mu}\mathcal{P}_2(\R^d)$ s.t. $\forall n\in \N$
\begin{align*}
    &\exp_{\mu}(\xi_n)=\exp_{\mu}(\tilde \xi_n)\\
    &\|\tilde \xi_n\|_{\mu}=W_2(\mu,\exp_{\mu}(\tilde \xi_n))=W_2(\mu,\exp_{\mu}(\xi_n))\leq \|\xi_n\|_{\mu}=\varepsilon_n.
\end{align*}
In other words, $\tilde \xi_n$ is an optimal velocity plan connecting $\mu$ and $\exp_{\mu}(\xi_n)$.
Since $p \in \partial^{-}_{M}\phi(\mu)$, there exists $N>0$ s.t.
\begin{equation*}
     \phi(\exp_{\mu}(\xi_n))-\phi(\mu)-\int_{\R^{2d}}\big<p(x),z\big>d\tilde \xi_n(x,z)\geq -\frac{1}{2}\delta\varepsilon_n, \quad \forall n>N.
\end{equation*}
By combining the two inequalities we deduce that 
\begin{align*}
    \big<(\Id\times p)_{\#}\mu, \frac{\text{Pr}(\tilde \xi_n)}{\varepsilon_n}\big>_{\mu}-\big<(\Id\times p)_{\#}\mu, \frac{\text{Pr}(\xi_n)}{\varepsilon_n}\big>_{\mu}&=
    \big<(\Id\times p)_{\#}\mu, \frac{\tilde \xi_n}{\varepsilon_n}\big>_{\mu}-\big<(\Id\times p)_{\#}\mu, \frac{\xi_n}{\varepsilon_n}\big>_{\mu}\\
    &\leq -\frac{\delta}{2}\varepsilon_n \quad \forall n>N,
\end{align*}
where we used the scalar product defined in \eqref{scalar} and its property along graphs.
We now observe that 
$$\left\|\frac{\text{Pr}(\tilde \xi_n)}{\varepsilon_n}\right\|_{\mu} \underbrace{\leq}_{\eqref{e: Pr}, \eqref{e: normW}
} \left\|\frac{\tilde \xi_n}{\varepsilon_n}\right\|_{\mu} , \quad  \left\|\frac{\text{Pr}(\xi_n)}{\varepsilon_n}\right\|_{\mu}\underbrace{\leq}_{\eqref{e: Pr}, \eqref{e: normW}
} \left\|\frac{\xi_n}{\varepsilon_n}\right\|_{\mu}$$
are both bounded by construction. We can then extract a weak limit in $L^2_{\mu}(\R^d)$ of the two, not relabeled, sequences $\left(\frac{\text{Pr}(\tilde \xi_n)}{\varepsilon_n}\right)_{n}, \left(\frac{\text{Pr}(\xi_n)}{\varepsilon_n}\right)_{n}$, that we call  $\tilde \xi, {\xi} \in L^2_{\mu}(\R^d)$, respectively. In particular, 

\begin{align}\label{e: contraddictionestimate}
   &\big<(\Id\times p)_{\#}\mu, \tilde \xi\big>_{\mu}-\big<(\Id\times p)_{\#}\mu,\xi\big>_{\mu}= \\
   &= \lim_{n \to \infty}\big<(\Id\times p)_{\#}\mu, \frac{\text{Pr}(\xi_n)}{\varepsilon_n}\big>_{\mu}-\big<(\Id\times p)_{\#}\mu, \frac{\text{Pr}( \tilde\xi_n)}{\varepsilon_n}\big>_{\mu} \nonumber \\
    &=\lim_{n \to \infty}\big<(\Id\times p)_{\#}\mu, \frac{\xi_n}{\varepsilon_n}\big>_{\mu}-\big<(\Id\times p)_{\#}\mu, \frac{\tilde{\xi}_n}{\varepsilon_n}\big>_{\mu}\leq -\frac{\delta}{2}. \nonumber
\end{align}
On the other hand, due to the construction of the plans $\xi_n, \tilde\xi_n$, we have that, for any $\alpha_n \in \Gamma_{\mu}(\xi_n,\tilde{\xi}_n)$, the equality
\begin{equation*}
\int_{\R^{3d}}\Big(\psi(z_0-x)-\psi(z_1-x)\Big)d\alpha_n(x,z_0,z_1)=\int_{\R^{3d}} \psi(y)d \exp_{\mu}(\xi_n)(y)-\int_{\R^{3d}} \psi(y)d \exp_{\mu}(\tilde\xi_n)(y)=0, 
\end{equation*}
holds $\forall \psi \in C^{\infty}_c(\R^d)$. Moreover, given $\psi \in C^{\infty}_c(\R^d)$, by a Taylor expansion of the function and its global semiconcavity, using the previous equality, there exists $C>0$ s.t.
\begin{equation*}
    0 \leq \int_{\R^{3d}} \big<\nabla \psi(x), z_0-z_1\big>d\alpha_n(x,z_0,z_1)+ C \big(\int_{\R^{2d}} |z_0|^2d\xi_n(x,z_0) _+ \int_{\R^{2d}} |z_1|^2d\tilde \xi_n(x,z_1)\big).
\end{equation*}
We then deduce 
\begin{align*}
   0&\leq \int_{\R^{3d}} \big<\nabla \psi(x), z_0-z_1\big>d\alpha_n(x,z_0,z_1) +C \big(\|\xi_n\|^2_{\mu}+  \|\tilde \xi_n\|^2_{\mu}\big)\\
   &\leq \big<(\Id \times \nabla \psi)_{\#}\mu,\xi_n\big>_{\mu}-\big<(\Id \times \nabla \psi)_{\#}\mu,\tilde\xi_n\big>_{\mu} +C \varepsilon^2_n\\
   &= \big<(\Id \times \nabla \psi)_{\#}\mu,\text{Pr}(\xi_n)\big>_{\mu}-\big<(\Id \times \nabla \psi)_{\#}\mu,\text{Pr}(\tilde\xi_n)\big>_{\mu}
\end{align*}
Dividing by $\varepsilon_n$ and passing to the limit $\varepsilon_n \to 0$, we obtain
\begin{equation*}
    \big<(\Id \times \nabla \psi)_{\#}\mu,\xi\big>_{\mu}-\big<(\Id \times \nabla \psi)_{\#}\mu,\tilde{\xi} \big>_\mu\geq 0.
\end{equation*}
Consider now the  $p \in \Tan^M_{\mu}\mathcal{P}_2(\R^d)$ given at the beginning of the proof: since $W_{\mu}((\Id \times p)_{\#}{\mu}, (\Id \times \nabla \psi)_{\#}{\mu})=\|p-\nabla \psi\|_{L^2_\mu(\R^d)}$, a density argument and the alternative description of the tangent space given in \eqref{Tangentspace} lead us to
$$
 \big<(\Id \times p)_{\#}\mu,\xi\big>_{\mu}-\big<(\Id \times p)_{\#}\mu,\tilde{\xi} \big>_\mu\geq 0,
$$
that together with \eqref{e: contraddictionestimate} gives the desired contraddiction.

The isometric relation \eqref{eq:last} follows from \eqref{isometryTangent}, and from the fact that, at absolutely continuous measures, $\text{Pr}(\partial^{-}_{S}\phi(\mu))=\partial_{S,M}^{-}\phi(\mu)$. 
\end{proof}

\begin{rmk}
In view of the previous lemma, we will frequently exploit the identification
provided by the isometric relation~\eqref{eq:last}. In particular, we shall not
distinguish between the measure 
\((\mathrm{Id}\times p)_{\#}\mu \in 
\partial^{-}_S \phi(\mu)\cap \Tan^K_{\mu}\mathcal{P}_2(\mathbb{R}^d)\)
and its projection
$
p := \mathrm{Pr}\big((\mathrm{Id}\times p)_{\#}\mu\big),
$
which is regarded as an element of
\(\partial^{-}_M \phi(\mu)\cap \Tan^{M}_{\mu}\mathcal{P}_2(\mathbb{R}^d)\).

\end{rmk}

\subsection{Selection of the differential} 
In the notion of viscosity solutions and from now on, given $\phi  \colon \mathcal{P}_2(\R^d) \to (-\infty, +\infty ]$ proper and lower semicontinuous and $\mu\in \text{Dom} (\phi)$ we adopt the following  differential
\begin{align}\label{e: differentialdefinitivo}
    \partial^{-}\phi(\mu):=\partial^{-}_S\phi(\mu)\cap \Tan^{K}_{\mu} \mathcal P _2(\R^d).
\end{align}
This definition extends to $\partial^{+}$ by the reverse operation $\gamma \mapsto -1 \cdot \gamma$, which, as remarked, leaves the tangent space invariant.

Note that the intersection $\partial^{-}\phi(\mu) \cap \partial^{+}\phi(\mu)$ contains at most one element since $\partial^{+}_S\phi(\mu)\cap\partial^{-}_S\phi(\mu)$ does.  Whenever this element exists, we say that the continuous function $\phi$ is differentiable at $\mu\in \mathcal{P}_2(\R^d)$.  This element, denoted with $\partial_{\mu}\phi$,  coincides with the Wasserstein gradient (and it is always a map,  otherwise  $\text{Pr}(\partial_\mu\phi)$ would be a different element in the intersection). As a consequence, if $\phi$ is fully $C^1$ then $\nabla \mathcal{D}_{\mu}\phi=\partial_{\mu}\phi$, see the structural results \cite[Lemma ~5.61 \& Theorem ~5.65]{CarmonaDelarue}.  

We end by stressing that thanks to Lemma \ref{Strong=Opt}, at absolutely continuous measures, our notion of subdifferential coincides, up to identification via the barycentric projection, with the one adopted in \cite{DS} in the definition of viscosity solutions of \eqref{HJBId}.
Moreover, being $\partial^{-}\subseteq \partial^{-}_{S,I}$, our subdifferential is a subset of the subdifferential introduced in \cite{BSOT} to treat viscosity solutions of \eqref{HJBDet}. 
\begin{rmk}[On the extrinsic and intrinsic geometry of the Wasserstein space]\label{ComparisonNotionofdiff} 
It is worth mentioning that we do not know if $\partial^{-}=\partial^{-}_K\cap \Tan^K_{(\cdot)}\mathcal{P}_2(\R^d)$, as in the Monge case, see Lemma \ref{Strong=Opt}. The answer to this question would give a correspondence between the extrinsic (e.g. $\partial^{-}_S$) and intrinsic (e.g. $\partial^{-}_K$) notions of differential, up to intersect them with the geometric tangent space.  We leave this question open for future research.

A similar result which explores the link between extrinsic and intrinsic geometry has been obtained in the recent paper \cite{CavSavSod23}, where an equivalence between geodesic convexity (intrinsic convexity) and strong (or total) convexity (extrinsic convexity) has been established for continuous functions, see Theorem 9.1 therein.
\end{rmk}

We state here an important result  that is 
fundamental in comparison arguments and describes the superdifferentiability of the Wasserstein distance squared that we need.\footnote{Actually, \cite[Proposition~2.4]{BSOT} is stated for $\partial^{+}_{S,I}$. However, a careful inspection of the simple proof reveals that the very same argument gives the result stated here also for our choice of superdifferential.}

\begin{prop}[See \cite{BSOT} Proposition 2.4]\label{prop:super}
For any $\mu,\nu \in \mathcal{P}(\R^d)$ and $\sigma \in \Gamma_{0}(\mu,\nu)$, consider the plan $\xi:=(\pi_1,\pi_1 -\pi_2)_{\#}\sigma$. Then the function $ \Phi : \eta \mapsto \frac{1}{2}W^2_2(\eta,\nu)$ is such that $\xi \in \partial^+\Phi(\mu).$
\end{prop}

We remark that a complete characterization of the elements in the superdifferential of the Wasserstein distance squared is given in \cite[Theorem. ~5.1.19]{aussedat2025optimal}.

\subsection{Main Hypotheses}

In what follows, and throughout the paper, we suppose the following

\begin{assumption}\label{a: DataHyp}
We assume that the Hamiltonian $H : \mathbb{R}^d \times \mathbb{R}^d \times \mathcal{P}_2(\mathbb{R}^d) \to \mathbb{R}$ satisfies the following properties:
\begin{itemize}
    \item[a)]  $H$ is globally continuous;

    \item[b)] there exists a constant $C>0$ such that, for all $p,q,x,y \in \mathbb{R}^d$ and $\mu,\nu \in \mathcal{P}_2(\mathbb{R}^d)$,
    \[
        |H(x,p,\mu)-H(y,q,\nu)|
        \le C \bigl(1 + |p| + |q|\bigr)
        \bigl( W_2(\mu,\nu) + |x-y| + |p-q| \bigr);
    \]

    \item[c)] there exists $C>0$
\begin{equation*}
|H(x,p,\mu)|\leq C(1+ |x|^2+|p|^2+ \mathcal{M}_2(\mu)), \quad \forall (x,p,\mu) \in \R^d\times \R^d \times \mathcal{P}_2(\R^d).
\end{equation*}
\end{itemize}

The terminal cost $\mathcal{G} : \mathcal{P}_2(\mathbb{R}^d) \to \mathbb{R}$ is assumed to be bounded and $W_1$--Lipschitz continuous.

The diffusion term in \eqref{HJBId} is of the form $a(\mu,x)=\varepsilon \text{Id}_{d\times d}$.

\end{assumption}

\subsection{Weak notion of convergence}\label{sectionweakconvergence}

The space $(\mathcal{P}_2(\mathbb{R}^d),W_2)$ is not locally compact, since the
underlying space $\mathbb{R}^d$ is non-compact
(see \cite[Rmk.~7.1.9]{AGS}).
This creates difficulties in the application of the doubling variable
technique, where one needs to extract a maximizing sequence; see
Lemma \ref{lem: comparison}.
For this reason, we introduce suitable continuity assumptions on the
functions under consideration.

\begin{defn}\label{weakconvergence}
We say that a sequence $(\mu_n)_n \subset \mathcal{P}_2(\mathbb{R}^d)$
\emph{weakly converges in $\mathcal{P}_2(\mathbb{R}^d)$} to
$\mu \in \mathcal{P}(\mathbb{R}^d)$ if $\mu_n$ converges narrowly to $\mu$
and the second moments are uniformly bounded, namely
\[
\sup_{n}\mathcal{M}_2(\mu_n)<\infty.
\]
\end{defn}

Note that, by the lower semicontinuity of the second moment with respect to narrow convergence,
we have
\begin{equation}\label{e:l.s.c.moments}
\mathcal{M}_2(\mu)\leq \liminf_{n\to\infty}\mathcal{M}_2(\mu_n)<\infty,
\end{equation}
for every sequence $(\mu_n)_n$ converging to $\mu$ in this topology.
In particular, this ensures that $\mu\in\mathcal{P}_2(\mathbb{R}^d)$.
In general, the inequality in~\eqref{e:l.s.c.moments} is strict; when
equality holds, one has $W_2(\mu_n,\mu)\to 0$ as $n\to \infty$, see
\cite[Theorem~6.9]{OTVillani}.

By the Prokorov Theorem \cite[Theorem ~5.1.3]{AGS}, the space $\mathcal{P}_2(\mathbb{R}^d)$,
endowed with this topology, is locally compact.\footnote{This follows from the tightness estimate
\[
\mu_n\bigl(\mathbb{R}^d \setminus B_R(0)\bigr)
\leq \frac{1}{R^2}\,\mathcal{M}_2(\mu_n),
\qquad \forall\, R>0.
\]
 } 
\begin{defn}
We say that a function
$U:\mathcal{P}_2(\mathbb{R}^d)\to\mathbb{R}\cup\{+\infty\}$ is
\emph{weakly lower semicontinuous in $\mathcal{P}_2(\mathbb{R}^d)$} if it
is lower semicontinuous with respect to the topology defined by the weak convergence in Definition \ref{weakconvergence}.
\end{defn}

We also observe that any function which is lower semicontinuous in
$\mathcal{P}_p(\mathbb{R}^d)$ for some $1\leq p<2$ is weakly lower
semicontinuous in $\mathcal{P}_2(\mathbb{R}^d)$.

Examples of functions that are lower semicontinuous with respect to this
notion of convergence include the Wasserstein distance
(see \cite[Prop.~7.1.3]{AGS}) and the entropy functional
(see \cite[Theorem~5.26]{OTVillani}, where lower semicontinuity follows
directly from its variational representation).

We conclude this subsection by emphasizing that, when one works on
$\mathcal{P}_2(\mathbb{T}^d)$, weak convergence and $W_2$ convergence
coincide.

\subsection{Definitions of viscosity solutions.}
\subsubsection{First order HJE}

In order to deal with the first-order HJE,  the author of \cite{BSOT}  considers the following relaxed Hamiltonian $\mathcal{H} \colon \mathcal{P}_2(\R^{2d})\to \R$ 
\begin{equation} \label{extendendHamiltonian}
 {\mathcal{H}}(\gamma):= \int_{\R^{2d}} H(x,z,{\pi_1}_{\#}\gamma)d\gamma(x,z). 
\end{equation}
This relaxation is reminiscent of the Kantorovich relaxation in the classical Optimal Transport theory and is the main ingredient of the comparison principle in \cite[Theorem 2.11]{BSOT}. We stress that the relaxed Hamiltonian coincides with the usual one at elements induced by maps, and, by the growth condition on $H$, 
\begin{equation}\label{growthH}
|\mathcal{H}(\gamma)|\leq C(1+\|\gamma\|^2_{(\pi_1)_{\#}\gamma}+ \mathcal{M}_2((\pi_1)_{\#}\gamma)), \quad \forall \gamma \in \mathcal{P}_2(\R^{2d}).\end{equation}

We are ready to give the notion of viscosity solutions for \eqref{HJBDet} by means of the differential defined in \eqref{e: differentialdefinitivo}, whose extension to the time dependent case is obvious.

\begin{defn}\label{D:solndet}
   We say that $U : [0,T] \times \mathcal{P}_2(\R^d) \rightarrow \R$ is a \emph{viscosity subsolution of} \eqref{HJBDet} 
	whenever for all $t \in [0,T), \mu\in \mathcal P_2(\R^d)$ and $(r,\gamma) \in  \del^+ U(t,\mu)$, we have
\begin{align*}
	-r + {\mathcal{H}}(\gamma) \leq 0.
\end{align*}
We say that $V : [0,T) \times \mathcal{P}_2(\R^d) \rightarrow \R$ is a \emph{viscosity supersolution of} \eqref{HJBDet} whenever for all $t \in [0,T], \mu\in \mathcal P_2(\R^d)$ and $(r,\gamma) \in  \del^-  V(t,\mu)$, we have
\begin{align*}
	-r + {\mathcal{H}}(\gamma) \geq 0.
\end{align*}
 We say that $U$ is a \emph{viscosity solution} of \eqref{HJBDet} if it is a viscosity subsolution and a viscosity supersolution.

 We say that $U: [0,T]\times \mathcal P_2(\R^d)\to \R$ is a classical sub(super)solution of \eqref{HJBDet} if $U$ is fully $C^1$ and 
 $$
 -\partial_tU(t,\mu)+ \int_{\R^d}H(x, \partial_\mu U, \mu)d\mu(x)\leq (\geq)0, \quad  \forall (t,\mu)\in [0,T)\times \mathcal{P}_2(\R^d).
 $$ 
\end{defn}

\begin{rmk}
We recall that the notion of viscosity solution introduced in \cite[Definition~2.8]{BSOT} is obtained from Definition~\ref{D:solndet} by replacing the differential $\partial$ with the larger set $\partial_{S,I}$.  
We emphasize that the notion of viscosity solution adopted in Definition~\ref{D:solndet} is \emph{weaker} than the one considered in \cite{BSOT}. Indeed, any subsolution (resp.\ supersolution) in the sense of \cite[Definition~2.8]{BSOT} is also a subsolution (resp.\ supersolution) in the sense of Definition~\ref{D:solndet}.

On the other hand, existence results for our notion of solution can be obtained more easily, since the definition of viscosity solution only needs to be tested against a smaller class of differentials.

Moreover, the two notions can be compared by means of the comparison
principle introduced in~\cite{BSOT}. In particular, one can show that
they coincide when the underlying space is the flat torus
$\mathbb{T}^d$, provided that existence holds for both notions of
solution. This equivalence is a direct consequence of
Proposition~\ref{prop:super}. An extension of this result to the whole
space $\mathbb{R}^d$ can be obtained by additionally assuming that both
solutions satisfy a weak continuity property. This assumption ensures
the existence of a maximum in the doubling variables argument.

\end{rmk}

\subsubsection{Semilinear HJE with idiosyncratic noise}

We start by recalling, for the sake of completeness, the definition of Entropy functional and Fisher information. We refer the reader to \cite{AGS} and \cite{OTVillani} for more details. 

The \emph{entropy functional} $\mathcal{E} \colon \mathcal P_2(\R^d) \to (-\infty, +\infty]$ is defined as
\begin{equation*}
\mathcal{E}(\mu):= 
\begin{cases}
    \int_\Rd  \log \mu(x)  
 d\mu(x), \quad &\text{if } \mu \in \mathcal{P}_{2,ac}(\R^d) \\
    +\infty \quad &\text{otherwise}.
\end{cases}
\end{equation*}
Here, and throughout the paper, we adopt the usual abuse of notation by denoting with $\mu$ also the density of the measure $\mu$ with respect to the Lebesgue measure $\lambda_d$, whenever it exists.
Having formally,  $\partial_\mu \mathcal{E}(\mu)=\nabla \mathcal D_\mu \mathcal{E}(\mu)= \nabla \log\mu(x)$  the \emph{Fisher information} $\mathcal{I} \colon \mathcal P_2(\R^d) \to [0,+\infty]$ is defined by 
\begin{equation*}
\mathcal{I}(\mu):= \begin{cases}
     \int_{\R^d}  \vert \nabla \log \mu (x) \vert^2 d\mu(x)= \int_{\R^d} \frac{\vert \nabla \mu (x)\vert^2}{\mu(x)}dx \quad &\text{if } \mu \in \mathcal{P}_{2,ac}(\R^d), \quad \frac{\vert \nabla \mu (x)\vert^2}{\mu(x)}\in L^1(\R^d) \\
    +\infty \quad &\text{otherwise}.
\end{cases}
\end{equation*}
For future use, we recall that 
\begin{equation}\label{eq: positivityEntropy}
    \mathcal E^{*}(\mu):=\mathcal E (\mu) + \pi \mathcal M_2 (\mu) \geq 0,
\end{equation}
for $\mu\in \mathcal{P}_2(\R^d)$, see \cite[Equation ~(2.5)]{DS}.

 For the definition of viscosity solution of \eqref{HJBId}, we follow  \cite{DS}. Note that the authors introduce  an entropic penalization that, thanks to the semilinear structure,  can be absorbed into the equation, up to a constant depending only on the magnitude of the penalization itself and the diffusion matrix.  
 
\begin{defn}[See {\cite[Definition~2.10]{DS}}]\label{D:soln} 
   We say that $U : [0,T] \times \mathcal{P}_2(\R^d) \rightarrow \R$ is a \emph{viscosity subsolution of} \eqref{HJBId} 	if  there is a constant $K>0$, depending only on $U,H$ and the diffusion $a$, such that, for every $\delta >0$,
whenever $t \in [0,T], \mu\in \mathcal P_2(\R^d)$, $\mathcal{I}(\mu)<\infty$, and $(r,p) \in \del^+ \left(U - \delta \mathcal{E} \right) (t,\mu)$, we have
\begin{align*}
	-r + \int_{\R^d} \langle \div_x a(\mu,x) + a(\mu,x)  \nabla \log \mu(x) , p(x)\rangle d\mu(x)  + \int_{\Rd} H \left(x, p(x), \mu \right)d\mu(x) \leq K \delta.
\end{align*}

We say that $V: [0,T] \times \mathcal{P}_2(\R^d) \rightarrow \R$ is a \emph{viscosity supersolution of} \eqref{HJBId}  if  there is a constant $K>0$, depending only on $V,H$ and the diffusion $a$, such that, for every $\delta >0$, whenever $t \in [0,T], \mu\in \mathcal P_2(\R^d)$, $\mathcal{I}(\mu)<\infty$, and $(r,p) \in \del^-\left( V + \delta \mathcal{E} \right) (t,\mu)$, we have
\begin{align*}
	-r + \int_{\R^d} \langle \div_x a(\mu,x) + a(\mu,x)\nabla \log \mu(x) , p(x) \rangle  d\mu(x) + \int_{\Rd} H \left(x, p(x), \mu \right)d\mu(x) \geq -K\delta.
\end{align*}
We say that $U$ is a \emph{viscosity solution} of \eqref{HJBId} if it is both a viscosity subsolution and a viscosity supersolution.

We say that $U: [0,T]\times \mathcal P_2(\R^d)\to \R$ is a classical sub(super)solution of \eqref{HJBId} if $U$ is partially $C^2$ and 
 $$
 -\partial_tU(t,\mu)-
\int_{\mathbb{R}^d} 
\mathrm{tr}\!\left[ a(\mu,x)\, \nabla \partial_\mu U \right] d\mu(x) +\int_{\R^d}H(x, \partial_\mu U, \mu)d\mu(x)\leq (\geq)0, \quad  \forall (t,\mu)\in [0,T)\times \mathcal{P}_2(\R^d).
 $$ 
 
\end{defn}
\begin{rmk}
In view of the computation done in  \cite[Section ~2.3, Eq, (2.23)]{DS}, the constant $K$ is expected to blow up as the diffusion goes to zero. We will take care of this fact in the proof of the vanishing viscosity limit. 
\end{rmk}

\section{Hopf-Lax formula}\label{s: Hopf}

The goal of this section is to establish an analogue of the finite-dimensional Hopf–Lax formula in the Wasserstein space. We show that this formula provides a representation of the viscosity solution to \eqref{HJBDet} in the sense of Definition~\ref{D:solndet}.
A first extension in this direction was obtained in \cite{GANGBO2019119} for the quadratic case, where the specific structure of the Hamiltonian is essential due to its relation with a projection operator. Here, we generalize their result beyond the quadratic setting, allowing for a convex Hamiltonian 
$H$ depending solely on the adjoint variable.
Moreover, in contrast with the assumptions in \cite[Section~5]{BSOT}, we do not require differentiability of the associated Lagrangian. In particular, the Hopf–Lax formula still yields a viscosity solution even when the regularity hypotheses ensuring existence in \cite{BSOT} are not satisfied.\\

Let us start by recalling the role of the Hopf-Lax formula in finite-dimensional spaces.   
We consider a  Lipschitz continuous datum $g:\R^d\to \R$ and a Hamiltonian $H:\R^d\to \R$ that depends only on the adjoint variable, convex, superlinear with at most quadratic growth, and has a minimum at $0$ (so that its Fenchel transform $L$ is such that $ L(0)=0$). 

Under these assumptions, it is well known that the function  $u: [0, T] \times \mathbb{R}^d \to \mathbb{R}$ defined by the Hopf-Lax formula: 
\[
u(t, x) = \inf_{y\in \R^d} \left\{ g(y) + (T-t)L\left(\frac{y - x}{T-t}\right) \right\}.
\]
is the viscosity solution 
of the HJE 
\[
\begin{cases}
-\partial_tw + H(\nabla w) = 0, \quad [0,T)\times\R^d \\
w(T, x) = g(x), \quad  x\in \R^d
\end{cases}.
\]
Note that in the representation formula the case $t=T$, gives $u(T,x)=g(x)$, by the superlinearity of the Lagrangian.

The above formula is formally derived using the method of characteristics. Its power also lies in the fact that it provides an alternative approach to obtain properties such as finite speed propagation of characteristics, Lipschitz continuity of the solution, and the propagation of semiconcavity provided the datum $g(\cdot)$ is semiconcave. Remarkably, in case the Hamiltonian $H$ is strictly convex, the semiconcavity of the solution at positive time is ensured even in the case of a non semiconcave datum. Let us underline that the key ingredient to prove that the function $u$, given by Hopf-Lax formula, is the viscosity solution of the HJE  is the semigroup structure given in the form of a Dynamic Programming Principle:
\begin{equation} \label{hopflaxmetric}
    u(t,x)=\min_{y \in \R^d} \left\{(s-t)L\left(\frac{y-x}{s-t}\right) +u(s,y)\right\}, \quad \forall\ 0\leq t<s\leq T.
\end{equation}
The aforementioned properties are standards for finite dimensional Euclidean spaces, see \cite[Section 3.3]{Evans}, and can be classically proven also for Hilbert spaces, see \cite{lasry_remark_1986} and references therein.

In this section, we make the following assumptions focusing on the Lagrangian, rather than on the Hamiltonian:
\newcounter{hopf}
\renewcommand{\thehopf}{\textup{Hopf\Alph{hopf}}}

\begin{itemize}
\refstepcounter{hopf}
\item[\thehopf] \label{item:HopfG}
The function $\mathcal{G}: \mathcal{P}_2(\R^d) \to \R$ is Lipschitz continuous.
\refstepcounter{hopf}
\item[\thehopf] \label{item:HopfL}
The Lagrangian $L: \R^{d} \to \R$ is convex, superlinear, and has at most quadratic growth, with $L(0)=0$.
Moreover, we assume the following coercivity condition: there exist constants $C, c > 0$ such that 
\[
L(z) \geq c |z|^2 - C \quad \forall\, z \in \R^d.
\]
\end{itemize}

\subsection{The relaxed Lagrangian as the Fenchel transform of the relaxed Hamiltonian }

We start by considering a relaxation of the  Lagrangian $\mathcal{L} \colon \mathcal{P}_2(\R^{2d}) \to \R$, defined by
\begin{equation}\label{def: mathcal L}
    \mathcal{L}(\gamma):=\int_{\R^{2d}} L\left(z\right) \gamma(dx,dz).
\end{equation}

We observe that the following property holds
\begin{equation}\label{coercivity Lagrangian MF}
\mathcal{L}(\gamma) \geq c \|\gamma\|^2_{(\pi_1)_\# \gamma}-C,
\end{equation}
for $C,c>0$ the same constant of the coercivity condition of the Lagrangian $L$.

The link between the relaxed Lagrangian and the relaxed Hamiltonian is explained by the following

\begin{lem}[Fenchel Transform of the Relaxed Hamiltonian]\label{l:Fenchel}
Let $\mathcal{H}:\mathcal{P}_{2}(\R^{2d})\to \R$ be the relaxed Hamiltonian associated to $H(p)=\max_{v\in \R^d}\big\{\big<-v,p\big>-L(v)\big\}$, i.e. $L^*=H$, defined in \eqref{extendendHamiltonian}. Fix $\mu \in \mathcal{P}_{2}(\R^d)$. Then, the equalities
\begin{align}\label{Fenchelduality}
    \mathcal{H}(\gamma)
    &=\max_{\xi \in \mathcal{P}_{2,\mu}(\R^{2d}), \theta \in \Gamma_{\mu}(\gamma,\xi)}\Big\{\int _{\R^{3d}}\big(\big<-v,p\big>-L(v)\big)d\theta(x,p,v)\Big\} \nonumber\\
    &=\max_{\xi \in \mathcal{P}_{2,\mu}(\R^{2d}),\theta \in \Gamma_{\mu}(\gamma,\xi)}\Big\{\big<-\xi,\gamma\big>_{\theta,\mu}-\mathcal{L}(\xi)\Big\} \nonumber\\
    &=\max_{\xi \in \mathcal{P}_{2,\mu}(\R^{2d})}\Big\{\big<-\xi,\gamma\big>_{\mu}-\mathcal{L}(\xi)\Big\}
\end{align}
hold for every choice of $\gamma \in \mathcal{P}_{2,\mu}(\R^{2d})$.

Moreover, for $\gamma \in \mathcal{P}_{2,\mu}(\R^{2d})$ 
\begin{equation}\label{Fenchel}
    \mathcal{H}(\gamma)+ \mathcal{L}(\xi)=\big<-\xi,\gamma\big>_{\mu} \Longleftrightarrow \xi=(\pi_1,-D_pH(\pi_2))_{\#}\gamma.
\end{equation}
In particular, 
\begin{equation}\label{Fenchelminus}
 \mathcal{H}(\gamma)=\max_{\xi \in \mathcal{P}_{2,\mu}(\R^{2d})}\Big\{-\big<\xi,\gamma\big>_{\mu}-\mathcal{L}(\xi)\Big\}.
\end{equation}
If, in addition, 
\begin{equation}\label{decreasing Fenchel}
\mathcal{L}(\xi)\geq \mathcal{L}(\mathrm{P}(\xi)), \quad \forall \xi \in \mathcal{P}_2(\R^{2d}),
\end{equation}
then, for all $ \gamma \in \Tan^K_{\mu}\mathcal{P}_2(\R^d)$,
\begin{equation}\label{FenchelTan}
\mathcal{H}(\gamma)= \max_{\xi \in \Tan^K_{\mu}\mathcal{P}_2(\mathbb{R}^d)}
\Big\{\langle -\xi,\gamma\rangle_{\mu}-\mathcal{L}(\xi)\Big\}=\max_{\xi \in \Tan^K_{\mu}\mathcal{P}_2(\mathbb{R}^d)}
\Big\{-\langle \xi,\gamma\rangle_{\mu}-\mathcal{L}(\xi)\Big\},
\end{equation}
In particular, any optimizer $\xi \in \mathcal{P}_{2,\mu}(\R^{2d})$ in \eqref{Fenchel}, with $\gamma \in\Tan^K_{\mu}\mathcal{P}_2(\R^d)$, is s.t. $\xi=\mathrm{P}(\xi)$, i.e. $\xi \in \Tan^K_{\mu}\mathcal{P}_2(\R^d)$ is the only maximizer in \eqref{FenchelTan}. 
\end{lem}
\begin{proof}
Let us prove the first equality, the second and the third one being immediate from this one.

Fix $\mu\in \mathcal P_2(\R^d)$ and let $\gamma \in \mathcal{P}_{2,\mu}(\R^{2d})$.
By the very definition of relaxed Hamiltonian \eqref{extendendHamiltonian} and Fenchel transform 
we have
\begin{align*}
    \mathcal{H}(\gamma)=\int_{\R^{2d}} H(p)d\gamma(x,p) &\underbrace{=}_{\theta \in \Gamma_{\mu}(\gamma,\xi)}\int_{\R^{3d}}H(p)d\theta(x,p,v)  \\
    & \underbrace{\geq}_{L^*=H } \int_{\R^{3d}} \big( \big<-v,p \big>-L(v)\big)d\theta(x,p,v), \quad \xi \in \mathcal{P}_{2,\mu}(\R^{2d}), \, \theta \in \Gamma_{\mu}(\gamma,\xi).
\end{align*}
Equivalently,  $$\mathcal{H}(\gamma)
    \geq \sup_{\xi \in \mathcal{P}_{2,\mu}(\R^{2d}), \theta \in \Gamma_{\mu}(\gamma,\xi)}\Big\{\int_{\R^{3d}}\big(\big<-v,p\big>-L(v)\big)d\theta(x,p,v)\Big\}.$$
    The equality holds iff $\theta$ is concentrated in {$\Gamma_{H}=\big\{(x,p,v)\in \R^{3d} : v = -D_p H(p)\big\}$}, where we denoted by $D_pH$ the differential of $H=L^*$, that is   $C^1$ because $L$ is strictly convex and superlinear.
 This implies that $\xi \in \mathcal{P}_{2,\mu}(\R^{2d}),\theta\in \Gamma_{\mu}(\gamma,\xi)$ optimize the first equality iff $\xi=(\pi_1,-D_pH(\pi_2))_{\#}\gamma$, and $\theta=(\pi_1, \pi_2,-D_pH(\pi_2))_{\#}\gamma$. This gives the assertion \eqref{Fenchel}. In particular, we also derive that for such a $\xi$
\begin{align*}
 \left< -\gamma, \xi \right>_{\mu}&=\max_{\theta \in \Gamma_{\mu}(\gamma,\xi)} \left< -\gamma, \xi \right>_{\mu, \theta}\underbrace{=}_{\Gamma_{\mu}(\gamma,\xi)=\{(\pi_1, \pi_2, -D_pH(\pi_2))_{\#}\gamma\}} \int_{\R^{2d}} \left<p, D_pH(p)\right>d\gamma\\
 &=-\int_{\R^{2d}} \left<p,- D_pH(p)\right>d\gamma=-\left<\gamma, \xi \right>_{\mu},
\end{align*}
that gives \eqref{Fenchelminus}.

Now, suppose the decreasing property \eqref{decreasing Fenchel}. For $\gamma \in \Tan^K_{\mu}\mathcal{P}_{2}(\R^d)$, we have
\begin{align*}
    \mathcal{H}(\gamma)&=\max_{\xi \in \mathcal{P}_{2,\mu}(\R^{2d})} \left\{ \left< -\gamma, \xi \right>_{\mu}- \mathcal{L}(\xi) \right\} \nonumber\\
&\underbrace{\leq}_{\eqref{decreasing Fenchel}, -\gamma \in \Tan^{K}_{\mu}\mathcal{P}_2(\R^d), \eqref{projectionOperator}} \sup_{\xi \in \mathcal{P}_{2,\mu}(\R^{2d})} \left\{ \left< -\gamma, \text{P}(\xi) \right>_{\mu}- \mathcal{L}(\text{P}(\xi)) \right\} \nonumber\\
    &=\sup_{\xi \in \Tan^K_{\mu}\mathcal{P}_2(\R^d)} \left\{ \left< -\gamma, \xi \right>_{\mu}- \mathcal{L}(\xi) \right\}.
\end{align*}
Moreover, if $\xi$ is an optimizer of \eqref{Fenchelduality} then $\text{P}(\xi)$ is an optimizer of \eqref{Fenchelduality} :
\begin{align*}
    \mathcal{H}(\gamma)+\mathcal{L}(\text{P}(\xi))&\underbrace{\leq}_{\eqref{decreasing Fenchel}} \mathcal{H}(\gamma)+\mathcal{L}(\xi)\underbrace{=}_{\xi  \text{ optimal}}\left<\xi,-\gamma \right>_{\mu}\\
    &\underbrace{=}_{\gamma \in \Tan^K_{\mu}\mathcal{P}_2(\R^d)} \left<\text{P}(\xi),-\gamma \right>_{\mu}\leq \mathcal{H}(\gamma)+\mathcal{L}(\text{P}(\xi)).
\end{align*}
This implies  $\mathcal{H}(\gamma)+\mathcal{L}(\text{P}(\xi))=\left<\text{P}(\xi),-\gamma \right>_{\mu}$, i.e. $\text{P}(\xi)$ is an optimizer, and $\mathcal{L}(\text{P}(\xi))=\mathcal{L}(\xi)$. By the characterization of optimizers $\xi=\text{P}(\xi)=(\pi_1,-D_pH(\pi_2))_{\#}\gamma$ and 
$$
\mathcal{H}(\gamma)=\max_{\xi \in \Tan^K_{\mu}\mathcal{P}_2(\R^d)}\left\{ \left< -\gamma, \xi \right>_{\mu}- \mathcal{L}(\xi) \right\}.
$$
\end{proof}

\subsection{Hopf-Lax Formula, Dynamic Programming Principle, regularity and infinitesimal behavior of optimizers}
For the sake of presentation, we derive the Hopf-Lax formula from the associated control problem. Consider the value function introduced in \cite{BSOT}

\begin{equation*} \label{valuefdet}
    U_K(t, \mu):= \inf_{\zeta_\cdot\in \mathcal{A}_{[t,T]}(\mu)} {\int_{t}^T\mathcal{L}(\zeta_s)ds  +\mathcal G((\pi_1)_{\#}\zeta_T)},
\end{equation*}
where for $t<T$ and $\mu \in \mathcal{P}_2(\R^d)$, $$\mathcal{A}_{[t,T]}(\mu):=\Big\{\zeta_{\cdot}\in C([t,T]; \mathcal{P}_2(\R^{2d})) : \partial_s \mu_s + \text{div}(\text{Pr}(\zeta_s)\mu_s)=0, \, \mu_s=(\pi_1)_{\#}\zeta_s \,\,\, \forall s \in [t,T],\, \mu_t=\mu \Big\}.$$

Let us show that this value function coincides with the classical value function $U_M$ defined as
$$
U_M(t,\mu):=\inf_{(\mu_{\cdot},\alpha_{\cdot}) \in \mathcal{A}^M_{[t,T]}(\mu)} {\int_{t}^T\int_{\R^d}L(\alpha_s(x))d\mu_s(x)ds  +\mathcal G(\mu_T)},
$$
where for $t<T$ and $\mu \in \mathcal{P}_2(\R^d)$, $$\mathcal{A}^M_{[t,T]}(\mu):=\Big\{(\mu_{\cdot},\alpha_{\cdot})\in C([t,T]; \mathcal{P}_2(\R^{2d})\times L^2_{\mu_{\cdot}}(\R^d;\R^d)) : \partial_s \mu_s + \text{div}(\alpha_s\mu_s)=0  \,\,\, \forall s \in [t,T], \, \mu_t=\mu\Big\}.$$
Indeed, we note that 
$$
U_{K}(t,\mu)\leq \inf_{\big\{\zeta_\cdot\in \mathcal{A}_{[t,T]}(\mu) \,:\, \zeta_s=\mu_s\otimes \delta_{\alpha_s} \big\}} {\int_{t}^T\mathcal{L}(\zeta_s)ds  +\mathcal G((\pi_1)_{\#}\zeta_T)}=U_M(t,\mu).
$$

On the other hand, by Jensen inequality $\mathcal{L}(\zeta)\geq \mathcal{L}((\Id \times \Pr(\zeta))_\# ((\pi_1)_\# \zeta)), \, \forall \zeta\in \mathcal{P}_{2}(\R^{2d})$.  Therefore, $\forall (t,\mu)\in[0,T]\times \mathcal{P}_2(\R^d)$

$$
U_{K}(t,\mu)\geq \inf_{\zeta_\cdot\in \mathcal{A}_{[t,T]}(\mu)} {\int_{t}^T\mathcal{L}((\Id \times \Pr(\zeta_s))_{\#} \mu_s)ds  +\mathcal G((\pi_1)_{\#}\zeta_T)}\geq U_{M}(t,\mu). 
$$
Therefore, $U_M=U_K$.
 
In view of the previous computations, the finite dimensional Hopf-Lax formula \eqref{hopflaxmetric} can naturally be generalized in the following way. 
\begin{defn}[Hopf-Lax formula]\label{def: Hopflax} Fix $\mathcal{G}\colon \mathcal{P}_2(\R^d)\to \R$, l.s.c. and bounded from below. We define the function 
\begin{align}\label{eq:hopf-lax1}
    \mathcal{V}(t,\mu):&=
    \inf_{\gamma \in \mathcal{P}_{2,\mu}(\R^{2d})}\left\{ \mathcal{G}((\pi_2)_{\#}(\gamma)) + (T-t)\mathcal{L}\left(\frac{1}{T-t}\cdot(\pi_1,\pi_2-\pi_1)_{\#}\gamma\right) \right\} \notag \\
\end{align}
where  $t\in [0,T], \,  \mu\in \mathcal P_2(\R^d)$, as the Hopf-Lax evolution of $\mathcal{G}$.
\end{defn}

\begin{rmk} \label{Existencemin}
Firslty, we observe that \eqref{eq:hopf-lax1} can be equivalently rewritten as 
 $$\mathcal{V}(t,\mu)=\inf_{\gamma \in \mathcal{P}_{2,\mu}(\R^{2d})}\left\{ \mathcal{G}(\exp_{\mu}(\gamma)) + (T-t)\mathcal{L}\left(\frac{1}{T-t}\cdot\gamma\right) \right\}.
 $$
 
Then, as in finite dimension, we also remark that thanks to the coercivity property \eqref{coercivity Lagrangian MF}, when $t = T$, we have
\begin{equation*}
\lim_{t \to T^{-}}(T-t)\mathcal{L}\left(\frac{1}{T-t}\cdot \gamma\right)=\begin{cases}
    0 \quad &\text{if }\gamma=\mu \otimes \delta_{0} \\
    + \infty \quad &\text{otherwise}
\end{cases}.
\end{equation*}
Therefore, $\mathcal{V}(T,\mu)=\mathcal G (\exp_\mu (\mu \otimes \delta_{0}))=\mathcal{G}(\mu)$.

Moreover, since $\mathcal{G}$ is bounded from below, $\mathcal{V}> -\infty$.  It is also immediate to see that the minimization can be split to obtain 
\begin{equation*}
    \mathcal{V}(t,\mu)= \inf_{\nu \in \mathcal{P}_2(\R^d) }\left\{ \mathcal{G}(\nu) + (T-t)\mathcal{L}_{T-t}(\mu,\nu) \right\},
\end{equation*}
where $\mathcal{L}_{T-t}(\mu,\nu):=\min_{\gamma \in \Gamma(\mu,\nu)}\Big\{\int_{\R^{2d}} L\left(\frac{y-x}{T-t}\right)d\gamma(x,y)\Big\}.$ 
\end{rmk}

Let us now prove that the function $\mathcal V$ defined in \eqref{eq:hopf-lax1} satisfies a semigroup property.

\begin{prop} [Dynamic Programming Principle] \label{DPP}
Let $L : \R^d \to \R$ satisfy \eqref{item:HopfL}. Then
\begin{align} \label{semigroup Hopf-Lax}
    \mathcal{V}(t,\mu)=
    &\inf_{\gamma \in  \mathcal P_{2,\mu}(\R^{2d}) } \left\{ \mathcal{V}(s, \exp_{\mu}(\gamma)) + (s-t)\mathcal{L}\left(\frac{1}{s-t}\cdot \gamma\right) \right\}, \quad \forall 
    \,0\leq t<s<T.
\end{align}
\end{prop}

\begin{proof}
Denote by $\mathcal{K}(t,\mu)$ the RHS of \eqref{semigroup Hopf-Lax}. In this proof, we use the subscript in a different way: $\gamma_t:=(\pi_1,\frac{\pi_2-\pi_1}{t})_{\#}\gamma,$ where $t>0$ and $\gamma \in \mathcal{P}_2(\R^{2d})$.
Then $\mathcal{K}(t,\mu)$ coincide with $$\inf_{\gamma \in  \mathcal P_{2,\mu}(\R^{2d})  } \left\{ \mathcal{V}(s, (\pi_2)_{\#}\gamma) + (s-t)\mathcal{L}(\gamma_{s-t}) \right\}, \quad \forall 
    \,0\leq t<s<T.$$
We then focus on proving the equality 
$$\mathcal V(t,
\mu)=\inf_{\gamma \in  \mathcal P_{2,\mu}(\R^{2d})  } \left\{ \mathcal{V}(s, (\pi_2)_{\#}\gamma) + (s-t)\mathcal{L}(\gamma_{s-t}) \right\} \quad \forall 
    \,0\leq t<s<T.$$

Fix $0\leq t<T$ and $\mu_1 \in \mathcal{P}_2(\R^d)$. For $t<s<T$ and $\gamma^{1 \,2}\in \mathcal{P}_{2,\mu_1}(\R^{2d})$, let us denote $\mu_2:=(\pi_2)_\# \gamma^{1\, 2}\in \mathcal P_2(\R^d)$ and use the formula of $\mathcal V$ given in Definition \ref{def: Hopflax}, so  that for 
$\varepsilon>0$, there exists $\gamma^{2\, 3}\in \Gamma(\mu_2,\mu_3)$ such that 
\begin{equation} \label{almost optimality Hopf-Lax}
    \mathcal{V}(s,\mu_2) \geq (T-s)\mathcal{L}\left(\gamma^{2\,3}_{T-s}\right)+\mathcal G(\mu_3) - \varepsilon,
\end{equation}
while, for every $\gamma\in \Gamma(\mu_1,\mu_3) $ it holds
 
\begin{equation}\label{eq:U1}
\mathcal{V}(t,\mu_1)\leq (T-t)\mathcal{L}(\gamma_{T-t})+\mathcal G(\mu_3).
\end{equation} 
Since $T>s>t$, we have that $\frac{T-s}{T-t} + \frac{s-t}{T-t}=1$ with addenda in $(0,1)$. 
Therefore, by the convexity of $L$, we obtain\begin{equation} \label{lagrangiana conv}
L\left(\frac{z-x}{T-t}\right)\leq \frac{s-t}{T-t}L\left(\frac{y-x}{s-t}\right)+\frac{T-s}{T-t}L\left(\frac{z-y}{T-s}\right).
\end{equation}
Now, let $\gamma^{1 \,3}\in \Gamma(\mu_1,\mu_3)$ be the composition of plans defined in \cite[ Remark 5.3.3]{AGS} as the plan given by $$\gamma^{1 \,3}:=\gamma^{2\,3}\circ\gamma^{1 \, 2}(dx,dz)=\int_{\R^d}\left(\gamma_{y}^{2\, 3}(dz) \times  \gamma_y^{1\,2}(dx) \right)\mu_2(dy) $$ where $\gamma^{1\,2}(dx,dy)=\gamma_y^{1\,2}(dx) \mu_2(dy)$ and $\gamma^{2\,3}(dy,dz)=\gamma_y^{2\,3}(dz) \mu_2(dy)$.
Then we can integrate \eqref{lagrangiana conv} with respect to $\gamma^{1 \,3}$ to obtain the following ``Triangular Inequality'' 
\begin{equation}\label{eq:triangular}
\mathcal{L}\left(\gamma^{1\,3}_{T-t}\right)\leq \frac{s-t}{T-t}\mathcal{L}\left(\gamma^{1 \, 2}_{s-t}\right)+\frac{T-s}{T-t}\mathcal{L}\left(\gamma^{2 \,3}_{T-s}\right).
\end{equation}
Thus combining \eqref{eq:U1} with \eqref{eq:triangular} and  using  \eqref{almost optimality Hopf-Lax}, we have
\begin{align*}
    \mathcal{V}(t,\mu_1)&\leq (s-t)\mathcal{L}\left(\gamma^{1 \, 2}_{s-t}\right)+(T-s)\mathcal{L}\left(\gamma^{2 \, 3}_{T-s}\right) +\mathcal G(\mu_3) \\
    &\leq(s-t)\mathcal{L}\left(\gamma^{1 \, 2}_{s-t}\right)+ \mathcal{V}(s, \mu_2) +\varepsilon
\end{align*}
By the arbitrariness of $\varepsilon$ we obtain $\mathcal{V}(t, \mu_1)\leq \mathcal K(t,\mu_1)$.

For the converse inequality we proceed as follows: fix $0\leq t<T$ and $\mu_1\in \mathcal P_2(\R^d)$ and 
let $\mu_3 \in\mathcal P_2(\R^d)$ and $\gamma^{1 \,3} \in \Gamma(\mu_1,\mu_3)$ be $\varepsilon$-optimizers for $\mathcal{V}(t,\mu_1)$.

For a fixed $t<s<T$, consider
$$\mu_2:=\left(\frac{T-s}{T-t}\pi_1+\left(1-\frac{T-s}{T-t}\right)\pi_2\right)_{\#}\gamma^{1 \, 3},$$
then, for $$\gamma^{2\,3}:=\left ( \frac{T-s}{T-t}\pi_1+\left(1-\frac{T-s}{T-t}\right)\pi_2, \pi_2\right )_\# \gamma^{1 \,3}\in \Gamma(\mu_2,\mu_3)$$ and $$\gamma^{1\,2}=: \left(\pi_1, \frac{T-s}{T-t}\pi_1+\left(1-\frac{T-s}{T-t}\right)\pi_2 \right)_\# \gamma^{1 \,3}\in \Gamma(\mu_1,\mu_2),$$  we have
\begin{align*}
   \mathcal K (t, \mu_1)\leq (s-t)\mathcal{L}\left(\gamma^{1 \, 2}_{s-t}\right)+\mathcal{V}(s,\mu_2)
   &\leq (s-t)\mathcal{L}\left(\gamma^{1 \,2}_{s-t}\right)+(T-s)\mathcal{L}\left(\gamma^{2 \, 3}_{T-s}\right) +\mathcal{G}(\mu_3)\\
   &=(T-t)\mathcal{L}\left(\gamma^{1 \, 3}_{T-t}\right) +\mathcal{G}(\mu_3)\leq \mathcal{V}(t,\mu_1) +\varepsilon,
\end{align*}
where we used the fact that $\mathcal L\left(\gamma^{1 \, 3}_{T-t}\right)=\mathcal L\left(\gamma^{2\, 3}_{T-s}\right)=\mathcal L\left(\gamma^{1\,2}_{s-t}\right)$.

Since $\varepsilon$ is arbitrary, we have the claim.

\end{proof}

\begin{lem}\label{LPHL}
 The function $\mathcal{V}$ defined in \eqref{eq:hopf-lax1} is Lipschitz continuous in $[0,\infty)\times \mathcal P_2(\R^d)$, provided that ${\mathrm{Lip}}(\mathcal{G}; W_2)< \infty$. More precisely, for all $0<t<T$
 \begin{equation*}
     \Lip(\mathcal{V}(t,\cdot); W_2)\leq \Lip(\mathcal{G};W_2) ,
\end{equation*}
and for all $\mu\in \mathcal P_2(\R^d)$
     \begin{equation*}
         \Lip_{t}(\mathcal{V}(\cdot,\mu)) \leq \sup \Big\{\mathcal{H}(\xi) \,|\, \xi \in \mathcal{P}_{2,\mu}(\R^{2d}), \|\xi\|_{\mu}\leq \Lip(\mathcal{G};W_2)  \Big\}.
 \end{equation*}

\end{lem}
\begin{proof}The proof is an adaptation of the proof of \cite[Lemma 2, Chapter 3] {Evans} to the Wasserstein framework.

Fix $0<t<T, \mu,\nu \in \mathcal{P}_2(\R^d)$. Fix $\varepsilon>0$, and choose $\gamma \in \mathcal{P}_{2,\mu}(\R^{2d})$ s.t.
\begin{equation*}
    \mathcal{G}(\exp_{\mu}(\gamma)) +(T-t)\mathcal{L}\left(\frac{1}{T-t}\cdot\gamma\right)\leq \mathcal{V}(t,\mu)+ \varepsilon,
\end{equation*}
and pick $\theta\in \Gamma_{0}(\gamma,\nu).$ We note that $\xi:=(\pi_3,\pi_2+\pi_3-\pi_1)_{\#}\theta \in \mathcal{P}_{2,\nu}(\R^d)$, and $\mathcal{L}(\xi_{T-t})=\mathcal{L}\left(\frac{1}{T-t}\cdot\gamma\right)$, where for $\xi_{T-t}$ we adopted the same notation used in the proof of the DPP, Proposition \ref{DPP}. 
Therefore,
\begin{align*}
    \mathcal{V}(t,\nu)-\mathcal{V}(t,\mu)&\leq\inf_{\zeta \in \mathcal{P}_{2,\nu}(\R^{2d})}\Big\{ \mathcal{G}((\pi_2)_{\#}\zeta) +(T-t)\mathcal{L}\left(\zeta_{T-t}\right)\Big\} -\mathcal{G}(\exp_{\mu}(\gamma))-(T-t)\mathcal{L}\left(\frac{1}{T-t}\cdot\gamma\right) + \varepsilon\\
    &{\underbrace{\leq}_{\zeta=\xi}} \mathcal{G}((\pi_2)_{\#}\xi)-\mathcal{G}(\exp_{\mu}(\gamma))+ \varepsilon\leq \Lip(\mathcal{G};W_2)W_2((\pi_2)_{\#}\xi,\exp_{\mu}(\gamma))+ \varepsilon\\
    &=  \Lip(\mathcal{G};W_2)\left(\int_{\R^{3d}
    } |z-x|^2d\theta(x,y,z)\right)^{\frac{1}{2}} + \varepsilon\\
    &=\Lip(\mathcal{G}; W_2)W_2(\mu,\nu)+\varepsilon.
\end{align*}

Since $\varepsilon$ is arbitrary, we can pass to the limit by letting $\varepsilon \to 0$, then inverting the role of the measures $\mu$ and $\nu$, we get the desired bound. In particular, $\Lip(\mathcal{V}(t,\cdot); W_2)$ is time independent.

We now prove the time Lipschitz estimate.  By the DPP, Proposition \ref{DPP}, fixing  $0<t<t+h<T$ for $ h>0$ and choosing $\gamma= (\mu \otimes \delta_0) \in \mathcal P_{2,\mu}(\R^{2d})$ we have 
\begin{equation}\label{value decreasing}
    \mathcal{V}(t,\mu)\leq \mathcal{V}(t+h,\mu) + h \mathcal{L}\left(\frac{1}{h}\cdot (\mu \otimes \delta_0)  \right)=\mathcal{V}(t+h,\mu),
\end{equation}
where we used the fact that $L(0)=0$.

On the other hand, using again the DPP and  $\Lip(\mathcal V(t,\cdot); W_2)\leq \Lip(\mathcal{G}; W_2)$, we have
\begin{align}
   \mathcal{V}(t,\mu)
   &=  \inf_{\zeta \in \mathcal{P}_{2,\mu}(\R^{2d})}\left\{\mathcal{V}(t+h,\exp_\mu (\zeta))  + h\mathcal{L}\left(\frac{1}{h}\cdot \zeta\right)\right\} \notag\\
   &\geq \inf_{\zeta \in \mathcal{P}_{2,\mu}(\R^{2d})}\left\{\mathcal{V}(t+h,\mu)  + h\mathcal{L}\left(\frac{1}{h}\cdot \zeta\right)-\Lip(\mathcal{G}; W_2)W_2(\mu,\exp_{\mu}(\zeta))\right\} \notag\\
   &=  \mathcal{V}(t+h,\mu) + \inf_{\zeta \in \mathcal{P}_{2,\mu}(\R^{2d})}\left\{h\mathcal{L}\left(\frac{1}{h}\cdot \zeta\right)-\Lip(\mathcal{G}; W_2)W_2(\mu,\exp_{\mu}(\zeta))\right\} \label{Liscthizbound1}.
\end{align}

In particular, we have by the one sided Cauchy Schwarz inequality
\begin{align}
    \Lip(\mathcal{G};W_2)W_2(\mu,\exp_\mu(\zeta))&\leq \Lip(\mathcal{G};W_2)\|\zeta\|_{\mu} \notag\\
    &=\max \Big\{\big<\xi,\zeta\big>_{\mu} : {\xi\in \mathcal{P}_{2,\mu}(\R^{2d}), \|\xi\|_{\mu}\leq \Lip(\mathcal{G};W_2)} \Big\} \notag\\
    &=\max \Big\{\big<-\xi,\zeta\big>_{\mu} : {\xi\in \mathcal{P}_{2,\mu}(\R^{2d}), \|\xi\|_{\mu}\leq \Lip(\mathcal{G};W_2)} \Big\} \label{Lipbound2}. 
\end{align}
The combination of \eqref{Liscthizbound1} and \eqref{Lipbound2} gives, 
\begin{align*}
    \mathcal{V}(t,\mu)&\geq \mathcal{V}(t+h,\mu) + \inf_{\zeta \in \mathcal{P}_{2,\mu}(\R^{2d})}\left\{h\mathcal{L}\left(\frac{1}{h}\cdot \zeta\right)-\sup \Big\{\big<-\xi,\zeta\big>_{\mu} : {\xi\in \mathcal{P}_{2,\mu}(\R^{2d}), \|\xi\|_{\mu}\leq \Lip(\mathcal{G};W_2)} \Big\}\right\}\\
    &= \mathcal{V}(t+h,\mu) -h \sup \left\{- \mathcal{L}\left(\frac{1}{h}\cdot \zeta\right)+ \big<-\xi,\frac{1}{h}\cdot \zeta\big>_{\mu} \, \big| \, \zeta, \xi \in \mathcal{P}_{2,\mu}(\R^{2d}), \|\xi\|_{\mu}\leq \Lip(\mathcal{G};W_2) \right\}\\
    &=\mathcal{V}(t+h, \mu)-h \sup \Big\{\mathcal{H}(\xi) \,|\, \xi \in \mathcal{P}_{2,\mu}(\R^{2d}), \|\xi\|_{\mu}\leq \Lip(\mathcal{G};W_2)  \Big\}.  
\end{align*}
Summarizing, 
\begin{equation*}
    0\underbrace{\leq}_{\eqref{value decreasing}}\mathcal{V}(t+h,\mu)-\mathcal{V}(t,\mu)\leq h \sup \left\{\mathcal{H}(\xi) \,|\, \xi \in \mathcal{P}_{2,\mu}(\R^{2d}), \|\xi\|_{\mu}\leq \Lip(\mathcal{G};W_2)  \right\}, 
\end{equation*}
where the RHS is finite thanks to the quadratic growth assumption on the Hamiltonian. 
\end{proof}

\begin{lem}[Infinitesimal behavior of optimizers]\label{Infinitesimal Behavior}
Let $0<t<t+h<T$ for $h>0$ and $\mu\in\mathcal P_2(\R^d)$. Fix $\varepsilon>0$, and
let $\xi \in \mathcal{P}_{2,\mu}(\R^{2d})$ be an $\varepsilon$-optimizer of $$
\mathcal{V}(t,\mu)=\inf_{\gamma \in \mathcal{P}_{2,\mu}(\R^{2d})}\left\{ \mathcal{V}(t+h,\exp_{\mu}(\gamma))+ h\mathcal{L}\left(\frac{1}{h}\cdot \gamma\right)\right\}.
$$ 
If $\mathrm{Lip}(\mathcal{G}; W_2)< \infty$, then there exists a constant $C>0$, depending only on $\mathrm{Lip}(\mathcal{G}; W_2)$ and on the coercivity parameters of $L$, s.t. $$\|\xi\|_{\mu}\leq C(h+ \sqrt{\varepsilon h}).$$
\end{lem}
\begin{proof}
Let $0<t<t+h<T$ for $h>0$ and $\mu\in\mathcal P_2(\R^d)$. Fix $\varepsilon>0$ and let $\xi \in \mathcal{P}_{2,\mu}(\R^{2d})$ be an $\varepsilon$-optimizer. We omit here the  dependence of $\xi$ on $h$. By the minimality property and \eqref{value decreasing}, we deduce 
\begin{equation*}
    \mathcal{V}(t+h,\exp_{\mu}(\xi))+ h\mathcal{L}\left(\frac{1}{h}\cdot \xi\right) \leq \mathcal{V}(t,\mu)  +\varepsilon\leq \mathcal{V}(t+h,\mu)  +\varepsilon,
\end{equation*} 
hence
\begin{equation*}
    0\leq \mathcal{V}(t+h,\mu)-\mathcal{V}(t+h,\exp_{\mu}(\xi)) - h\mathcal{L}\left(\frac{1}{h}\cdot \xi\right)+\varepsilon.
\end{equation*}
By the Lipschitz property in Lemma \ref{LPHL}, we have
\begin{equation*}
h\mathcal{L}\left(\frac{1}{h}\cdot \xi\right)\leq \Lip(\mathcal{V}(t+h, \cdot);W_2)W_2(\mu,\exp_{\mu}(\xi))+ \varepsilon \underbrace{\leq}_{ \text{Lemma }\ref{LPHL}} \Lip(\mathcal{G};W_2)W_2(\mu,\exp_{\mu}(\xi)) + \varepsilon .
\end{equation*}
Moreover, by the coercivity assumption on the Lagrangian, 
\begin{equation*}
    \-Ch +c\frac{1}{h}\|\xi\|^2_\mu\leq h\mathcal{L}\left(\frac{1}{h}\cdot \xi\right)\leq \Lip(\mathcal{G};W_2)W_2(\mu,\exp_{\mu}(\xi))+ \varepsilon\leq \Lip(\mathcal{G};W_2)\|\xi\|_\mu + \varepsilon.
\end{equation*}
By a standard absorbing argument we deduce 
$$\frac{1}{h}\|\xi\|^2_{\mu}\leq C(h+ \varepsilon),$$ with $C$ depending on $\mathcal G$ and on the growth parameters of $L$. Using the subadditivity of the square root, we conclude the claim.

\end{proof}

\subsection{Existence result}\label{existence}
We now show that the Hopf--Lax formula defines a viscosity solution, in
the sense of Definition~\ref{D:solndet}, to equation~\eqref{HJBDet}.

\subsubsection{Subsolution property}

Fix $(t,\mu)\in (0,T)\times \mathcal{P}_2(\R^d)$ and
let $(r,\gamma)\in \partial^{+}\mathcal V(t,\mu)$, namely $\forall \xi  \in \mathcal{P}_{2,\mu}(\R^{2d})$ and $\theta \in \Gamma_{\mu}(\gamma,\xi)$
\begin{equation*}
    \mathcal V(t+h,\exp_{\mu}(h\cdot \xi))-\mathcal V(t,\mu)\leq h\big<\gamma,\xi\big>_{\theta,\mu}+ rh + o(h\|\xi\|_\mu)+o(h), \quad \forall h>0 \text{ small enough}.
\end{equation*}
By the DPP, Proposition \ref{DPP}, we have 
\begin{equation*}
    \mathcal V(t,\mu)\leq \mathcal V(t+h,\exp_{\mu}(\tilde\xi)) + h\mathcal{L}\left(\frac{1}{h} \cdot \tilde \xi\right) \quad  \forall \tilde{\xi}\in \mathcal{P}_{2,\mu}(\R^{2d}).
\end{equation*}
In particular, by choosing $\tilde \xi=h\cdot \xi$ and combining the two previous  inequalities, we get
\begin{align*}
    0\leq h\mathcal{L}(\xi)+h\big<\gamma,\xi\big>_{\theta,\mu}+rh + o(h) +  o(h\|\xi\|_\mu), \quad \forall  \xi \in  \mathcal{P}_{2,\mu}(\R^{2d}), \theta \in \Gamma_{\mu}(\gamma,\xi), h>0 \text{ small enough}.
\end{align*}
Dividing by $h>0$ and sending $h \to 0$, we get
\begin{equation}\label{e: Fenchel Subsolution.}
   0\leq \mathcal{L}(\xi)+\big<\gamma,\xi\big>_{\theta,\mu}+r .
\end{equation}
Therefore, since $\xi\in \mathcal{P}_{2,\mu}(\R^{2d})$ and $\theta \in \Gamma_{\mu}(\gamma,\xi)$ are arbitrary,  we can invoke Lemma \ref{l:Fenchel} to obtain
\begin{equation*}
    0\leq - \mathcal{H}(\gamma)+r,
\end{equation*}
the desired inequality.

\begin{rmk}\label{decreasing}
The proof of the subsolution property motivates our choice of superdifferential. 
Indeed, in general, the Fenchel duality, stated in equation \eqref{Fenchelduality} of  Lemma~\ref{l:Fenchel}, does not 
appear to hold when the supremum is restricted to the tangent space. More precisely, for a fixed $\mu \in \mathcal{P}_2(\R^d)$, and for any
$\gamma \in \Tan^K_{\mu}\mathcal{P}_2(\R^d)$, we only have
\begin{align}\label{inequality Fenchel}
-\mathcal{H}(\gamma)
&\leq -\sup_{\xi \in \Tan^K_{\mu}\mathcal{P}_2(\mathbb{R}^d)}
\Big\{-\langle \xi,\gamma\rangle_{\mu}-\mathcal{L}(\xi)\Big\}\\
&= \inf_{\xi \in \Tan^K_{\mu}\mathcal{P}_2(\mathbb{R}^d)}
\Big\{\langle \xi,\gamma\rangle_{\mu}+\mathcal{L}(\xi)\Big\}\nonumber.
\end{align}

At present, we do not know whether equality holds for a general Lagrangian under  our standing assumptions. Nevertheless, under the additional 
invariance property for the Lagrangian
\begin{equation}\label{e: decreasingLagrangian}
\mathcal{L}(\mathrm{P}(\xi)) \leq \mathcal{L}(\xi),
\end{equation} the inequality in \eqref{inequality Fenchel} is indeed an equality as proven in Lemma \ref{l:Fenchel}. In this case,  we could have used, in the definition  of viscosity subsolution, the Kantorovich superdifferential, see Definition \ref{d: classicdiff},  intersected with the tangent space $\Tan^K_{\mu}\mathcal{P}_2(\R^d)$. Indeed, the inequality \eqref{e: Fenchel Subsolution.} would read as 
$$
0\leq \mathcal{L}(\xi)+\big<\gamma,\xi\big>_{\mu}+r, \quad \forall \xi \in \Tan^K_{\mu}\mathcal{P}_2(\R^d), 
$$
and, under the decreasing property \eqref{e: decreasingLagrangian} and by Lemma \ref{l:Fenchel},  
$$
0\leq -\mathcal H(\gamma)+r,
$$
concluding the argument.

We underline that this condition is 
reminiscent of the decreasing property along a projection operator introduced in
\cite[Equation (5.7), p.~27]{GANGBO2019119}, where the projection is $\text{Pr}_{\mu}\circ\text{P}: \mathcal{P}_{2,\mu}(\R^{2d}) \to \Tan^M_{\mu}\mathcal{P}_2(\R^d)$,  see \cite[Definition ~2.5]{GANGBO2019119}.
\end{rmk}

\subsubsection{Supersolution Property}\label{subsection: supersolution}

Fix $(t,\mu)\in (0,T)\times \mathcal{P}_2(\R^d)$ and let $(r,\gamma)\in \partial^{-}\mathcal V(t,\mu)$. Therefore, $\forall \xi  \in \mathcal{P}_{2,\mu}(\R^{2d})$  we have
\begin{align*}
     \mathcal V(t+h, \exp_{\mu}(\xi)) -\mathcal V(t,\mu)\geq \big<\gamma,\xi\big>_{\theta,\mu}+rh+ o(\|\xi\|_{\mu})+o(h), \quad \forall h>0 \text{ small enough.}
\end{align*}

For a fixed $h>0$, let $\xi \in \mathcal{P}_{2,\mu}(\R^{2d})$ be an $\varepsilon$-optimizer of \eqref{semigroup Hopf-Lax}, with $\varepsilon=h$ i.e.  $$\mathcal V(t,\mu) + h \geq \mathcal V(t+h,\exp_{\mu}(\xi))+ h \mathcal{L}\left(\frac{1}{h}\cdot \xi\right).$$ 
We do not stress here, as before, the dependence of $\xi$ on $h>0$. In any case, we know by Lemma \ref{Infinitesimal Behavior} that $\|\xi\|_{\mu}\leq 2Ch$, with $C$ specified therein. Therefore $o(\|\xi\|_{\mu})=o(h)$. 

Combining the two inequalities above at such a $\xi$ , we obtain 
\begin{equation*}
    -h\mathcal{L}\left(\frac{1}{h}\cdot \xi\right)\geq  h\big<\gamma,\frac{1}{h}\cdot \xi \big>_{\theta,\mu}+ rh + o(h). 
\end{equation*}
By the Fenchel inequality in Lemma \ref{l:Fenchel}, we have
\begin{equation*}
    -h\mathcal{L}\left(\frac{1}{h}\cdot \xi\right)+ h\big<\gamma, -\frac{1}{h}\cdot\xi\big>_{\theta,\mu}\leq  h\mathcal{H}(\gamma),
\end{equation*}
hence we obtain
\begin{equation*}
 -rh+h\mathcal{H}(\gamma)\geq o(h).
\end{equation*}
We conclude dividing by $h>0$ and passing to the limit as $h\to 0$.

\begin{rmk}
We leave for future work the analysis of the geodesic semiconcavity regularization property mentioned at the beginning of this section, in the case where the Lagrangian satisfies, for some \(\theta > 0\), the additional estimate
\[
\frac{1}{2}L(v_1)+\frac{1}{2}L(v_2)\leq L\!\left(\frac{v_1+v_2}{2}\right)+ \frac{1}{\theta}\lvert v_1-v_2\rvert^2, \qquad \forall\, v_1,v_2 \in \mathbb{R}^d,
\]
as in \cite[Lemma 4, p.~131]{Evans}. Another interesting direction for future research is to establish geodesic semiconcavity in the Wasserstein space via PDE-based estimates for viscosity solutions of both equations \eqref{HJBDet} and \eqref{HJBId}.

\end{rmk}

\section{Useful Estimates and Lemmata} \label{s: section3}

In this section, we derive some results that will be used in the main argument for the proof of the vanishing viscosity limit.

We first start by showing that, as a consequence of the Benamou-Brenier control representation, the elements of the subdifferential of $W_2$-Lipschitz functions are bounded.
\begin{lem}($L^2$ version of \cite[Lemma ~3.4]{DS}) \label{l: bounddiffpenali}
Let $U: \mathcal P_2(\R^d)\to \R$ be a $W_2$-Lipschitz function and set  $L=\Lip(U;W_2)< \infty$. Let  $\mu_0 \in \mathcal{P}_{2}(\R^d)$. 
\begin{itemize}
    \item Let $\delta > 0$, and $p\in\partial^{+}(U-\delta\mathcal{E})(\mu_0)$ (resp. $p \in \partial^{-}(U+\delta \mathcal {E})(\mu_0) $), then  $\| p + \delta \nabla \log \mu_0 \|_{L^2_{\mu_0}(\R^d)} \leq  L$ (resp. $\| p-\delta \nabla \log \mu_0 \|_{L^2_{\mu_0}(\R^d)} \leq  L $)  
    \item Let $\delta=0$, and $\xi \in \partial^{+}U(\mu_0)$ (resp. $\xi \in \partial^{-}U(\mu_0)$), then 
    $\| \xi \|_{\mu_0} \leq  L$.
\end{itemize}
\end{lem}

\begin{proof} 
We distinguish the two cases.
The first case is treated very similarly to the one given in \cite[Lemma ~3.4]{DS}, we only give a sketch of it in the superdifferential case. 
    
    Let $v: \R^d \to \R^d$ be smooth and bounded, and consider, for some $t_0 \geq 0$, the flow given by the continuity equation
	\begin{equation*}\label{v:cty}
		\del_t \mu_t + \div(v(x) \mu_t) = 0 \quad \text{in } (t_0,\oo), \quad \mu|_{t = t_0} = \mu_0.
	\end{equation*}
	In view of the smoothness of $v$, the solution  $[t_0,T] \ni t \mapsto \mu_t \in L^1(\Rd)$ is continuous. 
	
	Moreover, by the representation theorem for solutions of the continuity equation given in \cite[Thm.~8.3.1]{AGS} we have 
    \begin{equation*}\label{d2estimate}
W_2(\mu_t,\mu_{t_0})
\underset{\text{\cite[Box~5.1]{santambrogio2015optimal}}}{\le}
\int_{t_0}^t |\mu'|(s)\,{d}s
\underset{\text{\cite[Thm.~8.3.1]{AGS}}}{\le}
\int_{t_0}^t \|v\|_{L^2_{\mu_s}(\R^d)}\,{d}s,
	\end{equation*}
    where $ |\mu'|(s)$ denotes the metric derivative w.r.t. time in $(\mathcal{P}_2(\R^d),W_2)$ at time $s \in (t_0,t)$ of the absolutely continuous curve $t \mapsto \mu_t$.
	The previous bound and  $L=\Lip(U;W_2)<\infty$, give
 \begin{equation*}
     \vert U(\mu_t)-U(\mu_0) \vert \leq  L W_2(\mu_t,\mu_0) \le L \int_{t_0}^t \| v \|_{L^2_{\mu_s}(\R^d)}ds.
 \end{equation*}
 Moreover, by the renormalization property satisfied by $v$ and standard regularization and localization procedures (see \cite[Lemma ~3.4]{DS}), we also have
 \begin{equation*}
     \mathcal{E}(\mu_t)-\mathcal{E}(\mu_0)=-\int_{t_0}^{t} \int_{\R^d} \text{div} \,v(x)d\mu_s ds.
 \end{equation*}

	Fix now $h,\varepsilon > 0$ and $\varphi\in C_{c}^\infty(\R^d)$. By similar estimates performed in \cite[Lemma ~3.4]{DS}, it follows, from the superdifferential property and the Cauchy Schwarz inequality, that for all sufficiently small $h>0$, depending on $\|v\|_{L^2_{\mu_0}(\R^d)}$ and $\varepsilon$, 
	\begin{align*}
		&\int_{t_0}^{t_0+h} \int_\Rd \left( \langle \nabla \varphi(x), v(x)\rangle  - \delta \div v(x) \right) d \mu_s(x) \\
        &\geq  - L \int_{t_0}^{t_0+h} \|{v}\|_{L^2_{\mu_{s}}(\R^d)}ds - \|\nabla^2 \varphi\|_{L^2_{\mu_0}(\R^d)} o(h) - \norm{\nabla \varphi - p}_{L^2_{\mu_0}(\R^d)} \int_{t_0}^{t_0+h} \|{v}\|_{L^2_{\mu_{s}}(\R^d)}ds - \varepsilon h,    
	\end{align*}
    where $p \in\partial^{+}(U-\delta\mathcal{E})(\mu_0)$ has been taken as in the statement.

	Dividing by $h$ and sending $h \to 0^+$, due to the fact that  $v$ and $\varphi$ are  smooth and bounded, and $\mu_s \to \mu_0$ as $s \to t_0^+$ in $L^1(\Rd)$, we obtain
	\begin{align*}
		 \int_\Rd \left(\langle  \nabla \varphi(x), v(x)\rangle  - \delta \div v(x) \right)d \mu_0(x)
		\ge -  L\| v \|_{L^2_{\mu_0}(\R^d)} - \norm{\nabla \varphi - p}_{L^2_{\mu_0}(\R^d)} \|{v}\|_{L^2_{\mu_0}(\R^d)} - \varepsilon.
	\end{align*}
	Since $\varepsilon$ and $\varphi$ were arbitrary and $p \in \Tan^M_{\mu_0}\mcl P_2(\Rd)$, we infer that
	\[
		\int_\Rd \left( \big<p(x),  v(x)\big> - \delta \div v(x) \right)d \mu_0(x)
		\ge -  L\| v \|_{L^2_{\mu_{0}}(\R^d)}.
	\]
	In other words, by the arbitrariness of $v$, we conclude that the distribution
	\[
		p + \delta \nabla \log \mu_0 = p + \delta \frac{\nabla \mu_0}{\mu_0} 
	\]
	is bounded in $L^2_{\mu_0}(\R^d)$ with operator norm $L$. By the $L_{\mu_0}^2(\R^d)$ duality, this element belongs to $L^2_{\mu_0}(\R^d)$ and $\| p + \delta \nabla \log \mu_0 \|_{L^2_{\mu_0}(\R^d)} \leq  L$.\\
    
In the case $\delta=0$, we argue as follows. If $\xi \in \partial^{+}U(\mu_0)$, for $\gamma\in \mathcal P_{2,\bar \mu_0}(\R^{2d})$ and $\theta \in \Gamma_{\mu_0}(\xi, \gamma)$, 
\[
\|\gamma\|_{\mu_0} \, \Lip(U; W_2) 
\geq W_2(\mu_0, \exp_{\mu_0}(\gamma)) \, \Lip(U; W_2) 
\geq U(\exp_{\mu_0}(\gamma)) - U(\mu_0) 
\geq \langle \xi, \gamma \rangle_{\theta, \mu_0} + o(\|\gamma\|_{\mu_0}),
\]
which yields the claim by choosing \(\gamma = \frac{\varepsilon}{\|\xi\|_{\mu_0}} \, \xi\) with \(\varepsilon > 0\), dividing by \(\varepsilon\) and letting \(\varepsilon \to 0\).

\end{proof}
Viscosity solutions of \eqref{HJBDet} enjoy a time regularization.

\begin{lem}\label{l: Lipschitzregularityintime}
Let $U$ be a viscosity solution of \eqref{HJBDet}. Then the following implication holds
$$
L:=\sup_{t \in [0,T]} \Lip (U(t,\cdot); W_2) \Longrightarrow \sup_{\mu \in \mathcal{P}_2(\R^d)}\Lip_t(U(\cdot,\mu))\leq \sup_{\{\gamma \in \mathcal{P}_{2}(\R^{2d}): \|\gamma\|{_{(\pi_1)_\# \gamma}}\leq L\}}|\mathcal{H}(\gamma)|<\infty.
$$  
\end{lem}
\begin{proof}
It is sufficient to prove that each element in the time-differential is bounded by the constant $\sup_{\{\gamma \in \mathcal{P}_{2}(\R^{2d})\,:\, \|\gamma\|_{(\pi_1)_\# \gamma}\leq L\}}|\mathcal{H}(\gamma)|$. Firstly, we prove that this constant is finite. Indeed, by the growth assumption on $H$, using \eqref{growthH}, we have
$$
\sup_{\{\gamma \in \mathcal{P}_{2}(\R^{2d}): \|\gamma\|_{(\pi_1)_\# \gamma}\leq L\}}|\mathcal{H}(\gamma)| \leq \sup_{\{\gamma \in \mathcal{P}_{2}(\R^{2d}): \|\gamma\|_{(\pi_1)_\# \gamma}\leq L\}}C(1+ \|\gamma\|^2_{{(\pi_1)_\# \gamma}})\leq C(1+ L^2).
$$
Now fix $(t,\mu)\in [0,T)\times \mathcal{P}_2(\R^d)$ and consider $(r,\gamma)\in \partial^{+}U(t,\mu)$. By Definition \ref{D:solndet}, being $U$ a  viscosity subsolution of \eqref{HJBDet}, we have
$$
-|\mathcal{H}(\gamma)|\leq \mathcal{H}(\gamma)\underbrace{\leq}_{\text{Definition } \ref{D:solndet}} r. 
$$
Since $\gamma$ is an element of the superdifferential of a uniform in time Lipschitz function, we have 
$$
\|\gamma\|_{\mu}\leq L.
$$ 
Therefore, taking the infimum on a bigger set, we have
$$
\inf_{\{\gamma \in \mathcal{P}_{2,\mu}(\R^{2d}): \|\gamma\|_{\mu}\leq L\}}- |\mathcal{H}(\gamma)|\leq r \quad \forall r\in \partial_t^{+}U(t,\mu).
$$
With a similar computation, using the supersolution property, one can prove an analogous upper bound for the elements in the subdifferential. Namely,
$$
 r \leq \sup_{\{\gamma \in \mathcal{P}_{2,\mu}(\R^{2d}): \|\gamma\|_{\mu}\leq L\}} |\mathcal{H}(\gamma)|,\quad \forall r\in \partial_t^{-}U(t,\mu).
$$
\end{proof}
We now state a result that will be used in the proof of the vanishing viscosity limit.

\begin{lem}[Analogue of Proposition 3.1 of \cite{DS} in our context]\label{lem: comparison}
Let $U,V: \mathcal{P}_{2}(\R^d) \to \R$ be two bounded functions s.t. $\Lip(U; W_1)+ \Lip(V;W_2)<\infty$ and suppose $V$ is weakly l.s.c.. Fix $\alpha, \delta >0$, then the function $\Phi: \mathcal{P}_2(\R^d)\times \mathcal{P}_2(\R^d)\to [-\infty,+ \infty)$ defined by
\begin{equation*}
    \Phi(\mu,\nu):= U(\mu)-V(\nu)-\frac{1}{2\alpha}W_2(\mu,\nu)-\delta\mathcal{E}(\mu)-\delta\pi\mathcal{M}_2(\mu)-\delta\pi\mathcal{M}_2(\nu)
\end{equation*}
achieves its supremum. Moreover, any maximum point $(\bar \mu,\bar \nu)$ belongs to $\mathcal{P}_{2,ac}(\R^d)\times \mathcal{P}_{2}(\R^d)$, and
\begin{itemize}
    \item \label{1. Propo 3.1} setting \begin{equation*}
        p(x):=\frac{1}{\alpha}(x-T_{\bar \mu}^{\bar \nu}(x))+ 2\delta\pi x, \quad x\in \R^d
    \end{equation*}
      \begin{equation*}
        \xi:=(\pi_1,\frac{\pi_1-\pi_2}{\alpha}- (2 \delta\pi)\pi_1)_{\#}\sigma \in \Tan^K_{\bar \nu}\mathcal{P}_2(\R^d)
        \end{equation*} 
        where $\sigma \in \Gamma_{0}(\bar \nu,\bar \mu)$ is such that $(\pi_2,\pi_1)_{\#}\sigma= (\Id\times T_{\bar \mu}^{\bar \nu})_{\#} \bar \mu \in \Gamma_0(\bar \mu,\bar \nu),$ then $p \in {\partial^{+}}\big(U-\delta \mathcal{E}\big)(\bar \mu)$ and $\xi \in \partial^{-}V(\bar \nu)$. 
    Moreover,
    \begin{equation*}
        \|p+ \delta \nabla \log \bar \mu \|_{L^\infty(\R^d)}\leq \Lip(U;W_1), \, \|p+ \delta \nabla \log \bar \mu \|_{L^2_{\bar \mu}(\R^d)}\leq \Lip(U;W_2) \text{ and } \|\xi\|_{\bar\nu}\leq \Lip(V;W_2);
    \end{equation*}
    
    \item \label{2. Prop 3.1} $\bar \mu, \bar \mu^{-1} \in W^{1,\infty}_{\loc}(\R^d)$ and $\mathcal{I}(\bar \mu)< \infty$.
\end{itemize}
\end{lem}

\begin{proof}

The attainment of the maximum follows  the same argument used in the proof of \cite[Proposition~3.1]{DS}. 
The moment penalization ensures that any maximizing sequence 
\((\mu_n,\nu_n)\) remains in a compact subset of 
\(\mathcal{P}_2(\mathbb{R}^d)\times \mathcal{P}_2(\mathbb{R}^d)\) in the weak topology \eqref{weakconvergence}, more precisely $\sup_n \mathcal{M}_2(\mu_n)+ \mathcal{M}_2(\nu_n)< \infty$.
As a consequence, there exists a (not relabeled) subsequence 
\((\mu_n,\nu_n)\) converging weakly to some 
\((\bar\mu,\bar\nu) \in \mathcal{P}_2(\R^d)\times \mathcal{P}_2(\R^d)\).
By the assumed regularity of \(U\) and \(V\), together with the lower 
semicontinuity of the entropy and moments with respect to narrow convergence, the 
functional \(\Phi\) is upper semicontinuous with respect to the weak topology. Consequently, the limit pair \((\bar\mu,\bar\nu)\) is indeed a 
maximizer of \(\Phi\).

Let $(\bar \mu,\bar \nu)\in \mathcal{P}_2(\R^d)\times \mathcal{P}_2(\R^d)$ be a maximizer for $\Phi$. 
We note that $p\in \partial^{+}(U-\delta \mathcal{E})(\bar \mu)$ has been already proved in \cite[Proposition~3.1]{DS}.

The fact that \(\xi\) belongs to the subdifferential $\partial^{-}V(\bar \nu)$ follows directly from Proposition~\ref{prop:super}. Moreover, since \(2\pi\delta\nabla |x|^2 \in \Tan^M_{\bar\nu}\mathcal{P}_2(\R^d)\), we also have \(\xi \in \Tan^K_{\bar \nu}\mathcal{P}_2(\R^d)\). The bound $\|\xi\|_{\bar\nu}\leq \Lip(V;W_2)$ is a consequence of Lemma \ref{l: bounddiffpenali}.

We now focus on the regularity of the optimizer $\bar \mu$. The proof is similar to that done in  \cite[Proposition ~3.1]{DS}, and we just summarize the argument.  
Since $\Lip(U;W_1)< \infty$,  \cite[Lemma~3.3]{DS} yields $\bar \mu^{-1}\in L^{\infty}_{\text{loc}}(\R^d).$ Therefore, $T_{\bar \mu}^{\bar \nu}$ is locally bounded, being a.e.-$\bar \mu$ the gradient of an everywhere defined convex function. In particular, $p$ is locally bounded. So that, we have 
\begin{align*}
    &\begin{cases}
        p \in L^2_{\bar \mu}(\R^d)\\
        p + \delta \nabla \log \bar \mu\in L^2_{\bar \mu}(\R^d), \quad 
    \end{cases}
\Longrightarrow \nabla\log{\bar \mu} \in L^2_{\bar \mu}(\R^d), \text{ i.e. } \mathcal{I}(\bar \mu)<\infty \\
&\begin{cases}
        p \in L_{\loc}^{\infty}(\R^d)\\
        p + \delta \nabla \log \bar \mu\in L^\infty(\R^d)
\end{cases}
\Longrightarrow \nabla\log{\bar \mu} \in L^\infty_{\loc}(\R^d) \underbrace{\Longrightarrow}_{\bar \mu, \bar \mu^{-1} \in L^{\infty}_{\loc}(\R^d)} \bar \mu, \bar \mu^{-1} \in W^{1,\infty}_{\loc}(\R^d).
\end{align*}

\end{proof}

The following result will play a key role in obtaining a proper bound for the viscosity term in the proof of the vanishing viscosity limit.

\begin{lem}\label{BoundTermineViscoso}
Let $\mu \in \mathcal{P}_2(\R^d)$ s.t. $\mathcal{I}(\mu)<\infty$, $\mu \in W^{1,\infty}_{\mathrm{loc}}(\R^d)$ and $\mu^{-1}\in L^{\infty}_{\text{loc}}(\R^d)$. Fix $\nu \in \mathcal{P}_2(\R^d)$.  Then 
\begin{equation}\label{boundlaplaciandistancesquare}
    -\int_{\R^d}\big<\nabla \log \mu(x), x- T^\nu_{\mu}(x)\big>d\mu(x)\leq d.
\end{equation}
\end{lem}
\begin{proof}
From the regularity assumptions on $\mu$, it follows $\text{supp}(\mu)=\R^d$.
Therefore, it is sufficient to prove 
\begin{equation*}
    -\int_{\R^d}\big<\nabla \mu (x), x-T_{\mu}^{\nu}(x)\big>dx \leq d.
\end{equation*}
By Brenier Theorem, see \cite[Proposition ~3.1]{Bre}, and the fact that $\text{supp}(\mu)=\R^d$, we have that $T_{\mu}^\nu=\nabla \phi$ $\mu$-a.e, where $\phi : \R^d \to \R$ is a finite convex function. Being $\phi$ locally Lipschitz, $T^\nu_{\mu}$ is locally bounded. 
We then regularize the transport map by a convolution kernel $\rho_\alpha:=\alpha^{-d}\rho(\frac{x}{\alpha})$, where $\rho \in C^\infty(\R^d)$ is a probability density, compactly supported in the unitary ball $B_1(0)$ and $\alpha>0$. Set $T_{\alpha}:=\nabla \phi *\rho_{\alpha}=\nabla (\phi *\rho_\alpha)$, then $T_\alpha$ inherits the local boundedness of $T_{\mu}^{\nu}$, and $T_{\alpha}$ converges to $T_{\mu}^\nu$ in $L^2_{\text{loc}}(\R^d)$ (see \cite[Theorem ~6 pag. 630]{Evans}). Since convolution preserves convexity, $\phi_\alpha:=\phi *\rho_{\alpha}$ is convex. Therefore, the following inequality holds for every $\alpha>0$
\begin{equation*}
  \div(T_{\alpha}(x))=\text{tr}(D^2\phi_{\alpha}(x))\geq 0 \quad  \forall x \in \R^d.
\end{equation*}

We can then integrate this inequality against non-negative, compactly supported Lipschitz functions to obtain, for all  $\alpha>0$, 
\begin{equation*}
    -\int_{\R^d}\big<T_{\alpha}(x),
     \nabla \xi(x)\big> dx=\int_{\R^d}\div(T_{\alpha}(x))\xi(x) dx\geq 0, \quad \forall \xi \in W^{1,\infty}_c(\R^d) \text{ non negative.}
\end{equation*}
Since test functions are compactly supported, we can pass to the limit as $\alpha \to 0$ to obtain
\begin{equation}\label{inequalityconvexity}
    -\int_{\R^d}\big<T_{\mu}^\nu(x),
     \nabla \xi (x)\big> dx \geq 0, \quad \forall \xi \in W^{1,\infty}_c(\R^d) \text{ non negative}.
\end{equation}

Consider now for $R>0$,  $\chi_{R}(x):=\chi(\frac{x}{R})$, where $\chi \in C^{\infty}_c(\R^d)$ is constantly equal to 1 on $B_1(0)$, $\chi\leq 1$ on $B_2(0)\setminus B_1(0)$ and $\chi$ equal to zero on $\R^d\setminus B_2(0)$. 
We test inequality \eqref{inequalityconvexity} with the non negative function $\xi=\chi_R\mu \in W^{1,\infty}_c(\R^d)$ and we obtain 
\begin{align*}
     -\int_{\R^d} \big<x-T_{\mu}^\nu(x), \nabla \chi_{R} (x)\big> d\mu (x)  &-\int_{\R^d} \big<x-T_{\mu}^\nu(x), \nabla  \mu (x)\big> \chi_{R} (x)dx=\\&=- \int_{\R^d} \big<x-T_{\mu}^\nu(x), \nabla\xi(x)\big>dx \\& =\int_{\R^d}\text{div}(x)\xi (x) dx + \int_{\R^d}\big<T_{\mu}^{\nu}(x),\nabla \xi(x)\big>dx
     \\
    &\underbrace{\leq}_{\eqref{inequalityconvexity}} d. 
\end{align*}
On the other hand, we have the following bound
\begin{equation*}
   \int_{\R^d} \big<x-T_{\mu}^\nu(x), \nabla \chi_{R}(x) \big> d\mu(x) \leq \frac{C}{R}\Big(\int_{\R^d}|x-T_{\mu}^\nu(x)|^2d\mu(x)\Big)^\frac{1}{2}=\frac{C}{R}W_2(\mu,\nu)   
\end{equation*}
that combined with the previous inequality gives
\begin{equation*}
    -\int_{\R^d} \big<x-T_{\mu}^\nu(x), \nabla  \mu(x) \big> \chi_{R}(x) dx\leq \int_{\R^d} \big<x-T_{\mu}^\nu(x), \nabla \chi_{R}(x) \big> d\mu(x) +d \leq \frac{C}{R}W_2(\mu,\nu) + d.
\end{equation*}
Passing to the limit $R \to \infty$, we get the claim.

\end{proof}

\begin{rmk}\label{rmk: Rimannianbound}
The inequality \eqref{boundlaplaciandistancesquare} has a Riemannian flavour: In a $d$-dimensional smooth Riemannian manifold $(M,g)$ with non-negative Ricci curvature we have 
\begin{equation*}
    \Delta_{g} \frac{1}{2}d^2(x,y)\leq d,
\end{equation*}
where $d(x,y)$ is the geodesic distance between $x$ and $y$ in $M$. See \cite[Lemma ~42]{petersen2006riemannian}. In our context, the above inequality can be interpret as a weak analogue of 
\begin{equation*}
\Delta_{\mu}\frac{1}{2}W_2^2(\mu,\nu):=\int_{\R^d}\Delta_x (\mathcal D_\mu \frac{1}{2}W_2^2(\cdot,\nu)(\mu,x))d\mu(x) \leq d.
\end{equation*}
\end{rmk}

The inequality stated in Lemma \ref{BoundTermineViscoso} is a particular case of a more general result of independent interest, that holds at least when $d \geq 2$ for $\Lambda$-geodesically concave functions. Before stating this result, let us recall the following definition.
\begin{defn}
Given $\Lambda \in \R$, we say that $\mathcal{F}:\mathcal{P}_2(\R^d) \to \R\cup\{+ \infty\}$ is $\Lambda$-geodesically convex w.r.t. $W_2$, if for any $\mu_0,\mu_1 \in \text{Dom}(\mathcal{F})$ and  for every 
$W_2$-constant speed geodesic $t \mapsto \mu_t \in \text{Dom}(\mathcal{F})$ connecting  $\mu_0$ and $\mu_1$, $\mathcal F$ is $\Lambda$-convex along this geodesic, i.e. 
$$
\mathcal F(\mu_t)\leq (1-t) \mathcal F(\mu_0) + t\mathcal F(\mu_1)- \frac \Lambda 2 t(1-t)W_2^2(\mu_0,\mu_1), \quad \forall t\in[0,1].
$$
We say that $\mathcal{F}:\mathcal{P}_2(\R^d) \to \R\cup \{-\infty\}$ is $\Lambda$-geodesically concave w.r.t. $W_2$ if $-F$ is  $(-\Lambda)$-geodesically convex.
\end{defn}

\begin{prop}
    \label{generalconcavityestimate}
Suppose $d \geq 2$. Let $\mu \in \mathcal{P}_2(\R^d)$ s.t. $\mathcal{I}(\mu)<\infty$, $\mu^{-1}\in L^{\infty}_{\text{loc}}(\R^d)$ and $\mu \in W^{1,\infty}_{\text{loc}}(\R^d)$. Let $\mathcal{F}: \mathcal P_2(\R^d) \to \R$ be a $\Lambda$-geodesically concave functional and suppose $\mathcal{F}$ to be flat differentiable at $\mu\in\mathcal P_2(\R^d)$. Then $\R^d\ni x \mapsto \mathcal D_{\mu} \mathcal{F} (\mu,x)$ is $\Lambda$-concave, in particular $\R^d\ni x \mapsto \nabla \mathcal{D}_{\mu}\mathcal{F}(\mu,x)$ exists  $\mu$-a.e.. Moreover 
\begin{equation*}
    -\int_{\R^d}\big<\nabla \log \mu (x),\nabla \mathcal{D}_{\mu}\mathcal{F}(\mu,x)\big>d\mu(x)\leq d\Lambda.
\end{equation*}
\end{prop}   
Before proving the above result, let us introduce the context. Recall that the squared Wasserstein distance is $1$-geodesically concave, i.e.\ the map 
$\mu \mapsto W_2^2(\mu,\nu)$ is 1-concave along Wasserstein geodesics for every $\nu \in \mathcal{P}_2(\R^d)$.  Equivalently, $(\mathcal{P}_2(\R^d), W_2)$ is a geodesic PC space, see \cite[Definition 12.3.1]{AGS} and later discussions.
When $\mu$ is absolutely continuous, the flat derivative of $\mu \mapsto W_2^2(\mu,\nu)$ is given (up to an additive constant) 
by the optimal Kantorovich potential $\phi$ associated with the transport from $\mu$ to $\nu$.   
In particular, denoting by $T_\mu^\nu=\nabla \phi$ the optimal transport map, we have
\[
\mathcal D_\mu W_2^2(\cdot,\nu)(\mu,x)= \frac{|x|^2}{2}-\phi(x),
\]
that explains why the above  proposition is a generalization in dimension $d\geq 2$ of Lemma \ref{BoundTermineViscoso}.

Firstly, we derive, for functions on the Wasserstein space, a sufficient condition  to have a $\Lambda$-convex flat derivative, namely 

\begin{lem}\label{l:convexityalongdirc}
Let $\mathcal{F}: \mathcal{P}_2(\R^d)\to \R$ be a flat differentiable function at  $\mu \in \mathcal{P}_2(\R^d)$, s.t.  there exists $\Lambda=\Lambda(\mu)$ for which the following inequality holds
\begin{equation}\label{e:convexity along Dirac}
    \mathcal{F}((1-h) \mu+ h \delta_{y_t})\leq (1-t)\mathcal{F}((1-h)\mu+h\delta_{y_0}) + t\mathcal{F}((1-h)\mu+h\delta_{y_1})-h\Lambda \frac{t(1-t)}{2}|y_0-y_1|^2, 
\end{equation}
for all $y_0,y_1 \in \R^d$, with $y_{t}:=(1-t)y_{0}+ty_{1}, \, t,h \in [0,1].$

Then, $\R^d\ni x \mapsto \mathcal D_{\mu} \mathcal{F} (\mu,x)$ is $\Lambda$-convex.
\end{lem}
\begin{proof}
We derive the desired property from  the explicit formula given in \eqref{e:FlatderivativeRepresentiation} for the flat derivative.
Indeed, consider $y_{0},y_{1}$ and their convex interpolation $y_{t}=(1-t)y_{0}+ty_{1}$ for all $t\in[0,1]$.
Thanks to the assumption \eqref{e:convexity along Dirac} on $\mathcal{F}$, and the representation formula \eqref{e:FlatderivativeRepresentiation}, we have
 \begin{align*}\label{e: boundlambdaconve}
   \mathcal{D}_{\mu}\mathcal{F}(\mu, y_{t})&=\lim_{h\to 0}\frac{\mathcal{F}((1-h) \mu + h \delta_{y_t})-\mathcal{F}( \mu)}{h}\\
   &\leq (1-t)\lim_{h\to 0}\frac{\mathcal{F}((1-h) \mu + h \delta_{y_0})-\mathcal{F}( \mu)}{h} + t\lim_{h\to 0}\frac{\mathcal{F}((1-h) \mu + h \delta_{y_1})
   -\mathcal{F}( \mu)}{h}\\
   &-\lim_{h \to 0}\frac{h{\frac{\Lambda}{2} t(1-t) \vert y_0 -y_1 \vert^2}}{h}\\
   &= (1-t)\mathcal{D}_{\mu}\mathcal{F}( \mu, y_{0}) + t \mathcal{D}_{\mu}\mathcal{F}( \mu, y_{1}) -{\frac{\Lambda}{2} t(1-t) \vert y_0 -y_1 \vert^2}.\\
 \end{align*}
\end{proof}

Secondly, the following Lemma shows a geodesic convexity property along mixtures, i.e. flat convex combination of measures.

\begin{lem}\label{geodesic convexity} Suppose $d\geq2$.
Let $\mathcal{F}: \mathcal{P}_2(\R^d)\to \R$ be a continuous $\Lambda$-geodesically convex function. Let $\mu\in \mathcal P_2(\R^d)$, then
\begin{itemize}
    \item[(1)] $\forall h \in [0,1]$ the function 
    $$
    \nu \mapsto \mathcal{F}((1-h)\mu+ h\nu)
    $$
    is $h\Lambda$-geodesically convex.
    \item[(2)] In addition, if $\mathcal{F}$ is flat differentiable at $\mu$, then  Lemma \ref{l:convexityalongdirc} holds. 
\end{itemize}

\end{lem}

\begin{proof}
The proof relies on the result contained in  \cite[Proposition ~9.1]{CavSavSod23} that states: In dimension  $d\geq 2$, if $\mathcal{F} : \mathcal P_2(\R^d)\to \R$ is a continuous $\Lambda$-geodesically convex function, then it is \emph{totally} geodesically convex, namely given $\mu_0,\mu_1\in \mathcal{P}_2(\R^d)$ and {any} $\gamma\in \Gamma(\mu_0,\mu_1)$ we have 
\begin{equation*}
    \mathcal{F}(\mu_t)\leq (1-t)\mathcal{F}(\mu_0)+\mathcal{F}(\mu_1)-\frac{\Lambda}{2} t(1-t)(W_{2}^\gamma(\mu_0,\mu_1))^2 \quad \forall t\in [0,1],
\end{equation*}
where $\mu_t=(\pi_1+t(\pi_2-\pi_1))_{\#}\gamma$, and $(W_{2}^\gamma(\mu_0,\mu_1))^2:=\int_{\R^{2d}}|y-x|^2d\gamma(x,y)$ is defined in \eqref{eq: generalcost}.

Now, let $\mu\in \mathcal P_2(\R^d)$, given $ \nu_0,\nu_1 \in \mathcal P_2(\R^d)$, we  consider the mixture $\tau_{h,t}:=(1-h)\mu+ h \nu_t$, defined for $h\in [0,1]$ and $t\in [0,1]$, where $\nu_t:= (\pi_1+ t(\pi_2-\pi_1))_\# \tilde \gamma$ is the $W_2$-constant speed  geodesic induced by the optimal plan $\tilde \gamma \in \Gamma_{0}(\nu_0,\nu_1)$,  and choose the plan 
$$
\gamma_{h}:=(1-h)(\text{Id}\times \text{Id})_{\#}\mu+h \tilde\gamma.
$$
Observe that $\gamma_{h}\in \Gamma(\tau_{h,0},\tau_{h,1})$, and $\tau_{h,t}=(\pi_1+ t(\pi_2-\pi_1))_{\#}\gamma_{h}$.

Therefore, the total convexity property  of $\mathcal{F}$ yields
\begin{equation*}\label{convexitylemmamixture}
    \mathcal{F}(\tau_{h,t})\leq (1-t) \mathcal{F}(\tau_{h,0})+  t\mathcal{F}(\tau_{h,1})-\frac{\Lambda}{2} t(1-t)(W^{\tilde{\gamma}}_{2}(\tau_{h,0},\tau_{h,1}))^2,
\end{equation*}
then (1) follows  observing that 
$$
(W^{\tilde{\gamma}}_{2}(\tau_{h,0},\tau_{h,1}))^2=h W^2_2(\nu_0,\nu_1).
$$

Assertion (2) follows by choosing for all $y_0, y_1\in \R^d$ the $W_2$-constant speed geodesic $\nu_t=\delta_{y_t}$, where $y_t=(1-t)y_0+ ty_1$, connecting $\nu_0=\delta_{y_0}$ and $\nu_1=\delta_{y_1}$: indeed in this case $\nu_t=\delta_{y_t}$, and the $h\Lambda$-geodesically convexity of $\nu \mapsto \mathcal{F}((1-h)\mu + h \nu)$ restricted to Dirac masses corresponds to \eqref{e:convexity along Dirac}.

\end{proof}

\begin{rmk}
 
 Note that the main difficulty in the proof lies in the fact that, in general, the curve
\[
t \in [0,1] \longmapsto \tau_{h,t}=(1-h)\mu + h \nu_t
\]
is not necessarily a $W_2$-constant speed geodesic, even if \( t \mapsto \nu_t \) is one. The total convexity property, proved when $d\geq 2$ thanks to \cite[Proposition~9.1]{CavSavSod23}, avoid this issue.  Being our argument strongly based on that Proposition, we couldn't remove the constraint $d\geq 2$. 
\end{rmk}

As a byproduct of the previous lemmas, we obtain the following corollary.

\begin{cor}\label{geodesic convexity coro}
Suppose $d\geq 2$.
Let $\mathcal{F}: \mathcal{P}_2(\R^d)\to \R$ be a continuous $\Lambda$-geodesically convex function. If $\mathcal{F}$ is flat differentiable at $\mu\in \mathcal P_2(\R^d)$ then $\R^d \ni x \mapsto \mathcal{D}_{\mu}\mathcal{F}(\mu,x)$  is $\Lambda$-convex   
\end{cor}

We are now ready to prove the announced Proposition \ref{generalconcavityestimate}.
\begin{proof}[Proof of Proposition \ref{generalconcavityestimate}]
The argument of the proof follows the same lines of the proof of Lemma \ref{BoundTermineViscoso}. In that context, the 1-concave function $\frac{x^2}{2} - \phi(x)=\mathcal D_\mu \frac{W^2(\cdot,\nu)}{2}(\mu,x)$ is now replaced by the $\Lambda$-concave function $x\mapsto \mathcal{D}_{\mu} \mathcal{F}(\mu, x)$.
\end{proof}

We conclude this section with an application of the above results.

\begin{lem}\label{regularsubsolution}
Let $U^{\varepsilon}$ be a classical subsolution of \eqref{HJBId},  with $a(\mu,x)=\varepsilon \mathrm{Id}_{d \times d}$. 
If $U^\varepsilon$ is s.t. $\Delta_{x}\mathcal{D}_{\mu}U^{\varepsilon}(t,\mu,x)\leq C$ for some $C$ independent of $(t,\mu,x)\in [0,T]\times \mathcal{P}_2(\R^d)\times \R^d$, then $U^{\varepsilon}-\varepsilon C(T-t)$ is a classical subsolution of \eqref{HJBDet}. In particular, the inequality $\Delta_{x}\mathcal{D}_{\mu}U^{\varepsilon}(t,\mu,x)\leq C$ is satisfied with $C=d\Lambda$, provided that $d\geq 2$ and $U^{\varepsilon}$ is a $\Lambda$-geodesically concave functional.
\end{lem}
\begin{proof}
If $U^{\varepsilon}$ is a classical subsolution of \eqref{HJBId}, then $U^{\varepsilon}$ is partially $C^2$. In particular,  $\Delta_x\mathcal{D}_{\mu}U^{\varepsilon}(\mu,x)$  is well defined in $\text{supp}(\mu)$. Moreover,
\begin{align*}
-\partial_{t}U^{\varepsilon}+\mathcal{H}((\Id\times \partial_{\mu}U^{\varepsilon})_{\#}\mu)&=-\partial_{t}U^{\varepsilon}+\int_{\R^d}H(x,\partial_{\mu}U^{\varepsilon}, \mu)d\mu(x)\\
&\underbrace{\leq}_{U^{\varepsilon} \text{classical subsolution}} \varepsilon\int_{\R^d}\Delta_x \mathcal{D}_{\mu} U^{\varepsilon}(t,\mu,x)d\mu(x)\leq \varepsilon C.
\end{align*}

The second part of the statement is an immediate consequence of Corollary \ref{geodesic convexity coro}. 
\end{proof}

\section{Vanishing Viscosity Limit}\label{s: vvl}

In this section, we state and prove the vanishing viscosity result.

\begin{thm}\label{vanishing}

Suppose Assumptions \ref{a: DataHyp} hold. Let $U^{\varepsilon}$ and $U$ be two bounded viscosity solutions of  \eqref{HJBId} and \eqref{HJBDet}, respectively, s.t. that $ U^{\varepsilon}(T,\mu)=U(T,\mu)=\mathcal{G}(\mu)$.
Additionally, suppose $U^{\varepsilon}$ to be uniformly in time $W_1$-Lipschitz continuous for each $\varepsilon$, and $U$ to be uniformly in time $W_2$-Lipschitz continuous. Moreover, suppose $U$ to be jointly continuous with respect to the product topology of
$[0,T]\times\mathcal P_2(\mathbb R^d)$,
where $\mathcal P_2(\mathbb R^d)$ is endowed with the weak convergence as in Definition \ref{weakconvergence}.
Then, there exists a constant $C\geq 0$, depending only on the data, such that  
\begin{equation*} \label{rate}
\sup_{(t,\mu)\in[0,T]\times \mathcal{P}_2(\R^d)} \vert U^{\varepsilon}(t,\mu)- U(t,\mu) \vert \leq C\sqrt{\varepsilon}.
\end{equation*}
\end{thm}

\begin{proof}
The first part of the proof is mainly inspired by the proof of \cite[Theorem ~2.10]{Alekos}.

Firstly, we focus on the upper bound
\begin{equation}\label{error vanishing}
    E^+_{\varepsilon}:=\sup_{(t,\mu)\in [0,T]\times \mathcal{P}_2(\R^d)} U^{\varepsilon}(t,\mu)-U(t,\mu).
\end{equation}
It will be clear through the proof that a symmetry argument, inverting the role of the two functions, will give also the bound for 
$$
E_{\varepsilon}^{-}:=\sup_{(t,\mu)\in[0,T]\times \mathcal{P}_2(\R^d)}U(t,\mu)-U^{\varepsilon}(t,\mu).
$$
\\

{\it Step 1: Error quantification term.}\\

 We want to prove that
$$
E^{+}_{\varepsilon}\leq C\sqrt{\varepsilon}.
$$
We note that $E^{+}_{\varepsilon}$ is bounded by the boundedness assumptions on the two functions $U^\varepsilon$, $U$. 
Clearly, if $E^{+}_{\varepsilon}\leq 0$ then \eqref{error vanishing} holds. 
Therefore, we assume $E^{+}_{\varepsilon}>0$.\\

{\it Step 2: Doubling variables penalization.}\\

We invoke the usual doubling variables argument by considering the penalized function
\begin{align*}	
\Phi_{\delta,\alpha}(s,t,\mu,\nu) := & \, U^{\varepsilon}(s,\mu) -\delta( \mathcal{E}(\mu) + \pi \mathcal{M}_2(\mu)) - U(t,\nu) -\delta \pi \mathcal{M}_2(\nu)-\frac{1}{2\alpha} W_2^2(\mu,\nu)
		 \\
   &\! - \frac{1}{2\alpha}|t-s|^2 -\frac{2T-t-s}{4T}E^{+}_{\varepsilon},
\end{align*}
	over $[0,T] \times [0,T]  \times \mathcal{P}_2(\R^d) \times \mathcal{P}_2(\R^d)$. Here $\alpha,\delta>0$, with $\alpha$ to be chosen in terms of $\varepsilon$.\\

By adapting Lemma~\eqref{lem: comparison} to functions dependent on time, the function $\Phi_{\delta,\alpha}$
attains its maximum. We denote by
$(\bar s,\bar t,\bar \mu,\bar \nu)\in [0,T]^2\times\mathcal{P}_2(\mathbb{R}^d)^2$
a point where this maximum is achieved.\\

Now, being $\mathcal{E}^*:=\mcl E+ \pi\mathcal{M}_2\ge 0$, we have that the function 
	\[
		\delta \mapsto M_{\delta,\alpha} := \max_{(s,t,\mu,\nu) \in [0,T]^2 \times  (\mcl P_2(\R^d))^2} \Phi_{\delta,\alpha}(s,t,\mu,\nu)
	\]
	is nonincreasing. Therefore, for fixed $\alpha> 0$, the limit
	\[
		M_{\alpha} := \lim_{\delta \to 0} M_{\delta,\alpha}
	\]
    exists.  Moreover, from the inequality
	\begin{align*}
		M_{\delta,\alpha} = \Phi_{\delta,\alpha}(\bar s, \bar t, \bar \mu, \bar \nu)\leq M_{\delta/2,\alpha} - \frac{ \delta }{2}( \mcl E(\bar \mu) +\pi\mathcal{M}_2(\bar \mu)) -\frac{\delta \pi}{2}\mathcal{M}_2(\bar \nu).
	\end{align*}
    we deduce 
    \begin{equation*}
        \frac{ \delta }{2}( \mcl E(\bar \mu) +\pi\mathcal{M}_2(\bar \mu)) +\frac{\delta \pi}{2}\mathcal{M}_2(\bar \nu)\leq M_{\delta/2,\alpha} - M_{\delta,\alpha} \xrightarrow{\delta \to 0} 0 
    \end{equation*}
 for fixed $\alpha > 0$. In particular,
	\[
		\lim_{\delta \to 0} \delta \mcl E^*(\bar \mu) = 0.
	\]
	Moreover, since $\mathcal{E}^{*}$ satisfies $\mcl E^*(\bar \mu)\geq  \frac{\pi}{2} \mathcal{M}_2(\bar \mu) -\frac{d}{2}\log 2$, for fixed $\alpha> 0$,  
	\begin{equation}\label{moments:small}
	 \lim_{\delta \to 0}\delta (\mathcal{M}_2(\bar \mu)+ \mathcal{M}_2(\bar \nu))=0.
	\end{equation}
Therefore, both entropy and moment penalizations vanish with $\delta$. \\

Now, rearranging the terms in the inequality $ \Phi_{\delta,\alpha}(\bar s, \bar t, \bar \mu, \bar \mu)  \le \Phi_{\delta,\alpha}(\bar s, \bar t, \bar \mu, \bar \nu)$, leads to the inequality
	\begin{align*}
		\frac{1}{2\alpha} W^2_2(\bar \mu, \bar \nu) &\le  U^{\varepsilon}(\bar s, \bar \mu) - U^{\varepsilon}(\bar s, \bar \mu) +  U(\bar t, \bar \mu) -  U(\bar t, \bar \nu)\\
        &=  U(\bar t, \bar \mu) -  U(\bar t, \bar \nu)
	\end{align*}
 By the uniform in time $W_2$-Lipschitz regularity of $ U$, we recover the following bound 

\begin{equation*}
		\frac{1}{2\alpha} W^2_2(\bar \mu, \bar \nu) \leq \Lip(U; W_2) W_2(\bar \mu, \bar \nu) .
	\end{equation*}
In particular,
\begin{equation} \label{diagonal: smallt}W_2(\bar \mu,\bar \nu)\leq 2 \alpha \Lip(U; W_2).
\end{equation}
We have a similar estimate in time as a consequence of the following inequality $\Phi_{\delta,\alpha}(\bar s, \bar t, \bar \mu, \bar \nu)  \le \Phi_{\delta,\alpha}(\bar s, \bar s, \bar \mu, \bar \nu)$ that yields
\begin{equation*}
    \frac{|\bar t-\bar s|^2}{2\alpha}\leq U(\bar t, \bar \nu)-U(\bar s,\bar \nu)+\frac{\bar t-\bar s}{4T}E_{\varepsilon}^{+},
\end{equation*}
which implies, thanks to the uniform Lipschitz continuity in time of $U$, ensured by Lemma \ref{l: Lipschitzregularityintime}, and the boundedness of $E_{\varepsilon}^{+}$,
\begin{equation}\label{time boundcomparison}
    |\bar t-\bar s|\leq C \alpha.
\end{equation}\\

We now consider the three cases: either $\bar s=T$, or $\bar t=T$, or $\bar t,\bar s<T.$\\

{\it Step 3: Error estimate at the terminal time.}\\

The case $(\bar s,\bar t) \in \{T\}\times [0,T].$ We consider the inequality $\Phi_{\delta,\alpha}(t,t,\mu,\mu)\leq \Phi_{\delta,\alpha}(\bar s, \bar t, \bar \mu,,\bar \nu)$ $\forall (t,\mu)\in [0,T]\times \mathrm{Dom}(\mathcal{E})$, we get 
\begin{align*}
    U^{\varepsilon}(t,\mu)-U(t,\mu)-\delta \mathcal{E}(\mu)-2\delta \pi \mathcal {M} (\mu)&\leq \mathcal{G}(\bar \mu)-U(\bar t, \bar \nu)+ \frac{2T-2t}{4T}E_{\varepsilon}^+\\
    &\leq  \mathcal{G}(\bar \mu)-\mathcal{G}(\bar \nu)+ \mathcal{G}(\bar \nu)-U(\bar t, \bar \nu) +\frac{2T-2t}{4T}E_{\varepsilon}^+\\
    &\leq \Lip(\mathcal{G};W_2)W_2(\bar \mu,\bar \nu)+\Lip_t(U)|T-\bar t|+\frac{2T-2t}{4T}E_{\varepsilon}^+\\
    &\underbrace{\leq}_{\eqref{diagonal: smallt},\eqref{time boundcomparison}}C\alpha+\frac{1}{2}E_{\varepsilon}^+, 
\end{align*}
for all $(t,\mu) \in [0,T]\times \mathrm{Dom}(\mathcal{E})$. Having obtained a bound from above that does not depend on $\delta$ we can send it to zero. Therefore,
$$
 U^{\varepsilon}(t,\mu)-U(t,\mu)\leq C\alpha+\frac{1}{2}E_{\varepsilon}^+ \quad  \forall(t,\mu)\in [0,T]\times \mathrm{Dom}(\mathcal{E}). 
$$
Passing to the supremum in the LHS and recalling that $\overline{\mathrm{Dom}(\mathcal{E})}^{W_2}=\mathcal{P}_2(\R^d)$ we get
$$
E_{\varepsilon}^+\leq C \alpha+ \frac{1}{2}E_{\varepsilon}^{+}.
$$
Finally, the choice $\alpha=\sqrt{\varepsilon}$ gives the claim in this scenario.

The case $(\bar s,\bar t) \in [0,T]\times  \{T\}$ is completely analogous.\\

{\it Step 4: The case $(\bar s,\bar t)\in [0,T) \times [0,T)$: Application of the definitions of viscosity solution.}\\

By adapting Lemma \ref{lem: comparison} to functions dependent on time, it follows
 \begin{equation*}
 \begin{cases}
     &(\frac{\bar s- \bar t}{\alpha} - \frac{1}{4T}E_{\varepsilon}^{+}, \mu \otimes \delta_{\frac{1}{\alpha}(x-T_{\bar{\mu}}^{\bar{\nu}}(x)) +2\pi \delta x}) \in \partial^+( U^{\varepsilon}-\delta \mathcal{E})(\bar s, \bar \mu)\\
     & (\frac{\bar s- \bar t}{\alpha}+\frac{1}{4T}E_{\varepsilon}^{+}, \bar \Sigma) \in \partial^- U(\bar t, \bar \nu),
\end{cases}     
 \end{equation*}
 where $\bar \Sigma=(\pi_1,\frac{\pi_1-\pi_2}{\alpha} - 2 \pi \delta \pi_2 )_{\#}\bar \sigma$, with $\bar \sigma \in \Gamma_{0}(\bar \nu,\bar \mu)$ s.t. $(\pi_2,\pi_1)_{\#}\bar\sigma=(\Id\times T_{\bar \mu}^{\bar \nu})_{\#}\bar \mu$.
Moreover, we have $\mcl I(\bar \mu)< \oo$. Therefore, we can apply the Definitions \ref{D:soln} and \ref{D:solndet} to obtain the following inequalities 
	\begin{align*}
	\frac{1}{4T}E_{\varepsilon}^{+}-\frac{\bar s- \bar t}{\eta} &+ \int_{\R^d}  H \left(x,\left(\frac{1}{\alpha}(x-T_{\bar{\mu}}^{\bar{\nu}}(x)) +2\pi \delta x\right), \bar{\mu} \right)d\bar{\mu}(x) \\
		&+\varepsilon \int_{\R^d} \nabla \log \bar{\mu}(x) \cdot \left( \frac{1}{\alpha}\left(x-T_{\bar{\mu}}^{\bar{\nu}}(x)\right) +2\pi \delta x \right)d\bar \mu(x) \leq K(\varepsilon)\delta,
	\end{align*}
	and
	\begin{align*}
		- \frac{1}{4T}E_{\varepsilon}^{+} {-}\frac{\bar s-\bar t}{\eta}&+  \int_{\R^d \times \R^d } H \left(y, \frac{1}{\alpha}(x-y) - 2 \pi\delta  y , \bar\nu\right)d\bar \sigma (y,x) \geq 0.
	\end{align*}

Here, we call $K(\varepsilon)$ the penalization constant in the Definition \ref{D:soln}, stressing its possible dependence on $\varepsilon$.
Subtracting the two inequalities, we obtain the following.
	\[
		\frac{1}{2T}E_{\varepsilon}^{+} \le K(\varepsilon) \delta + I_1 + I_2 + I_3,	
	\]
  where    
	\begin{align*}
		I_1 := \int_{\R^d \times \R^d}  \left[H \left( y,\frac{1}{\alpha}(x-y) + 2 \pi \delta y , \bar{\nu} \right) -   H \left( x, \left(\frac{1}{\alpha}(x-y) +2\pi \delta x\right), \bar{\mu} \right)\right] d\bar{\sigma}(y,x),
	\end{align*}
	\begin{align*}
		I_2 &:=  -\varepsilon \int_{\R^d} \nabla \log \bar{\mu}(x) \cdot \frac{1}{\alpha}(x-T_{\bar{\mu}}^{\bar{\nu}}(x))   d\bar \mu(x),
	\end{align*}
	and
	\begin{align*}
		I_3 := -2\pi \delta \varepsilon \int_{\R^d} \nabla \log \bar{\mu}(x) \cdot  x  d\bar \mu(x). 
	\end{align*}

    In what follows, the constant $c$ denotes a constant that changes from line to line, and is independent of $\alpha,\delta$ and $\varepsilon$.\\

{\it Step 5: Estimate of the integrals $I_1 \& I_3.$}\\

Using Cauchy Schwarz inequality and \eqref{moments:small}, \eqref{diagonal: smallt} we get 
\begin{align*}
    |I_1|&\leq c \int_{\R^d \times \R^d}\Big(1+\left|\frac{x-y}{\alpha}\right|+\delta |x| + \delta |y|\Big)\Big(|x-y|+\delta |y|+ \delta|x|+W_2(\bar \mu,\bar \nu)\Big)d\bar \sigma \\
    &\leq c W_2(\bar \mu,\bar \nu) + \frac{W^2_2(\bar \mu,\bar \nu)}{2\alpha} + \delta (\mathcal{M}^\frac{1}{2}(\bar \mu) +\mathcal{M}^\frac{1}{2}(\bar \nu))+ \delta (\mathcal{M}(\bar \mu)+ \mathcal{M}(\bar \nu))\\
    &\underbrace{\leq}_{\eqref{diagonal: smallt},\eqref{moments:small}} c {\alpha}+ \delta^{\frac{1}{2}}+ \delta (\mathcal{M}(\bar \mu)+ \mathcal{M}(\bar \nu)).
\end{align*}

The third integral $I_{3}$ can be bounded thanks to the fact that $\bar \mu \in W_{\loc}^{1,\infty}$and $ \bar \mu^{-1} \in  L_{\loc}^\infty$. Indeed, in this case, as in the proof in \cite[Theorem ~2.14]{DS}, we can perform an integration by parts and prove
\begin{equation*}
I_3= 2\pi\delta \varepsilon \int_{\R^d} \div(x)d\bar \mu(x) =  2\pi d\delta \varepsilon\leq c\delta\varepsilon.
\end{equation*}

{\it Step 6: Estimate of the viscosity term.}\label{step: estim}\\

Adapting Lemma \ref{lem: comparison} to functions depending on time, the measure $\bar \mu$ has enough regularity to invoke Lemma \ref{BoundTermineViscoso}. Therefore, we obtain
\begin{equation*}  
I_2=-\varepsilon\int_{\R^d}\big<\nabla \log \bar \mu, \frac{x- T^{\bar \nu}_{\bar \mu}(x)}{\alpha}\big>d\bar \mu\leq \frac{ \varepsilon}{\alpha}d.
\end{equation*}

{\it Step 7: Conclusion.}\\

In summary, the previous estimates give
 \begin{equation*}
\frac{1}{2T}E_{\varepsilon}^+  \leq K(\varepsilon)\delta + c(\alpha + \frac{\varepsilon}{\alpha}) + \delta^{\frac{1}{2}} +\delta(\mathcal{M}(\bar \mu)+ \mathcal{M}(\bar \nu)),
\end{equation*}
where the constant $c$ is independent of $\varepsilon, \alpha, \delta$.
Now we pass to the limit $\delta \to 0$ and, recalling the vanishing behavior of the penalizations \eqref{moments:small}, we get

$$\frac{1}{2T}E_{\varepsilon}^+ \leq c (\alpha+ \frac{\varepsilon}{\alpha})$$

We then optimize w.r.t. $\alpha$, so that $\alpha=\sqrt{\varepsilon}$.
With this choice of the  we finally get 
$$
\frac{1}{2T}E_{\varepsilon}^+\leq c \sqrt{\varepsilon}.
$$

By the considerations at the beginning of the proof, the  same argument allows to obtain the bound
$$
E_{\varepsilon}^{-}\leq C \sqrt{\varepsilon}.
$$

\end{proof}

\begin{rmk}
The boundedness assumption on $U$ and $U^\varepsilon$ can be relaxed to a more general growth condition, as already observed in \cite{DS}. We also observe that the additional continuity of $U$ w.r.t. the weak convergence can be dropped in the case the equation is considered in $\mathcal{P}(\T^d)$. 
\end{rmk}

\begin{rmk}
As in the finite-dimensional setting, the rate of convergence follows from the regularity assumptions on the Hamiltonian: modifying its local modulus of continuity directly affects the convergence speed. The argument can be adapted to this framework without additional difficulty. Finally, we can also handle the case of terminal costs depending on the viscosity parameter $\varepsilon$ thanks to the stability property established in \cite[Theorem ~4.1]{DS}.  
\end{rmk}
\begin{rmk}
The above method yields an optimal convergence rate. For a Hamiltonian satisfying Assumption \ref{a: DataHyp}, this rate is the best one can expect without assuming convexity of $H$, since the finite-dimensional setting is embedded in this framework. It would be interesting to investigate whether this bound can be improved in the convex case, as independently shown in \cite{cirant2025convergenceratesvanishingviscosity, chaintron2025optimalrateconvergencevanishing}.
\end{rmk}

We end this section with a one-sided improvement of the previous bound in a more restrictive case. 

\begin{prop}
Under the assumption of Theorem \ref{vanishing}, if
 $U^{\varepsilon}$ is a classical subsolution of \eqref{HJBId} s.t. $\Delta_x \mathcal D_\mu U^{\varepsilon}(t,\mu,x)\leq C$ for some $C$ independent of $(t,\mu,x)\in [0,T]\times \mathcal{P}_2(\R^d)\times \R^d$.
Then 
\begin{equation*}
    U^{\varepsilon}(t,\mu)-U(t,\mu)\leq C(T-t)\varepsilon, \quad \forall (t,\mu)\in [0,T]\times \mathcal P_2(\R^d).
\end{equation*}
In particular, if $d \geq 2$, and $\mu \mapsto U^{\varepsilon}(t,\mu)$ is geodesically $\Lambda$-concave, the result holds with $C=d\Lambda.$
\end{prop}
\begin{proof}
It is sufficient to observe that we can apply Lemma \ref{regularsubsolution} and then the comparison principle in \cite[Theorem ~2.11]{BSOT}.
\end{proof}

\section{Acknowledgments}

This research is supported by the project King Abdullah University of Science and Technology Research Funding (KRF) under award no. CRG2024-6430.6.  and by the project PRIN 2022 (prot. 2022W58BJ5)  “PDEs and optimal control methods in mean field games, population dynamics and multi-agent models". While this work was written GCS and DT were associated to INdAM and the group GNAMPA. DT was partially supported by ``iNEST: Interconnected Nord-Est Innovation Ecosystem"" funded under the National Recovery and Resilience Plan (NRRP), Mission 4 Component 2 Investment 1.5 - Call for tender No. 3277 of 30 December 2021 of Italian Ministry of University and Research funded by the European Union - NextGenerationEU, Project code: ECS00000043, Concession Decree No. 1058 of June 23, 2022, CUP C43C22000340006. The authors thank Averil Aussedat, Charles Bertucci and Alekos Cecchin for the fruitful discussions.

\bibliographystyle{plain} 
\bibliography{refs}

\end{document}